\documentclass[10pt]{amsart}
\usepackage{fancyhdr}
\usepackage[english]{babel}
\usepackage{amssymb}
\usepackage{mathrsfs}
\usepackage{enumerate}
\usepackage{color}
\usepackage{ifpdf}
\usepackage{tikz}
\usepackage{tikz-cd}
\usetikzlibrary{decorations.pathmorphing,arrows}
\usepackage[initials]{amsrefs}

\usepackage{color}
\usepackage{amsmath}
\usepackage{amssymb,amsfonts}
\usepackage{graphicx}
\usepackage{times,amssymb,amscd}
\usepackage{verbatim}
\usepackage{enumerate}
\usepackage{subcaption}

\DeclareMathOperator\arctanh{arctanh}

\usepackage{lscape}

\newcommand{\bbH}{\mathbb{H}}

\newcommand{\bbR}{\mathbb{R}}

\newcommand{\bh}{\mathbf{h}}
\newcommand{\ba}{\mathbf{a}}

\newcommand{\Ba}{\boldsymbol{a}}
\newcommand{\Bb}{\boldsymbol{b}}

\newcommand{\Bu}{\boldsymbol{u}}

\newcommand{\cD}{\mathcal{D}}
\newcommand{\cF}{\mathcal{F}}

\newcommand{\cT}{\mathcal{T}}
\newcommand{\cC}{\mathcal{C}}
\newcommand{\cB}{\mathcal{B}}

\newcommand{\Hor}{\text{Hor}}

\newtheorem{theorem}{Theorem}
\newtheorem{corollary}{Corollary}
\newtheorem{lemma}{Lemma}
\newtheorem{defn}{Definition}
\newtheorem{rmrk}{Remark}

\usepackage{tikz, tikz-3dplot, pgfplots}
\usetikzlibrary[positioning,patterns] 

\usepackage{pictexwd,dcpic}

\usepackage{lscape}

\begin{document}

\title[New Lower Bounds for Optimal Horoball Packing Density\dots]{New Lower Bounds for Optimal Horoball Packing Density in Hyperbolic $n$-space for $6 \leq n \leq 9$}

\author{Robert Thijs Kozma}
\address{Department of Mathematics, Statistics, and Computer Science\\
 University of Illinois at Chicago \\
Chicago IL 60607 USA \\
rthijskozma@gmail.com}

\author{Jen\H{o}  Szirmai}
\address{ Budapest University of Technology and Economics\\
Institute of Mathematics, Department of Geometry \\
H-1521 Budapest, Hungary\\
szirmai@math.bme.hu}
\thanks{\includegraphics[height=4mm]{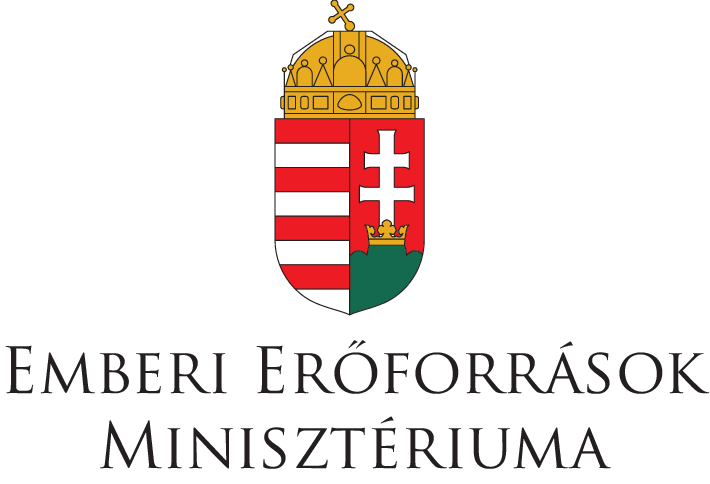}
		Supported by the \'UNKP-18-3 New National Excellence Program of the Hungarian Ministry of Human Capacities}

\date{}

\maketitle

\begin{abstract}
Koszul type Coxeter simplex tilings exist in hyperbolic $n$-space $\mathbb{H}^n$ up to $ n = 9$, and their horoball packings achieve the highest known regular ball packing densities for $n = 3, 4, 5$. In this paper we determine the optimal horoball packing densities of Koszul simplex tilings in dimensions $6 \leq n \leq 9$, which give new lower bounds for optimal packing density in each dimension. The symmetries of the packings are given by Coxeter simplex groups, and a parameter related to the Busemann function gives an isometry invariant description of different optimal horoball packing configurations.

\end{abstract}


\section{Introduction}

This is the fourth paper in a series we determine the optimal horoball packing densities of Koszul-type noncompact Coxeter simplex tilings that exist in $\bbH^n$ for $2 \le n \le 9$. 
In \cite{KSz,KSz14,KSz18} we considered dimensions $ 3 \leq n \leq 5$ respectively and in the present paper we consider dimensions $6 \le n \le 9$.

First, in \cite{KSz}, we showed that the classical example of the horoball packing in $\overline{\bbH}^3$ that achieves the B\"or\"oczky-type simplicial packing density upper bound $d_3(\infty)$ (cf. Theorem \ref{t:Boroczky}) by tiling by a regular ideal simplex is not unique, and gave several new examples using horoballs of different types. 
Second, in \cite{KSz14}, we found seven horoball packings of Coxeter simplex tilings in $\overline{\mathbb{H}}^4$ that yield densities of $5\sqrt{2}/\pi^2 \approx 0.71645$, counterexamples to L. Fejes T\'oth's conjecture for the maximal packing density of $\frac{5-\sqrt{5}}{4} \approx 0.69098$ in his foundational book {\it Regular Figures} \cite[p. 323]{FTL}.
Finally, in \cite{KSz18} we constructed the densest known ball packing in $\overline{\mathbb{H}}^5$ with a density of $\frac{5}{7 \zeta(3)}$ where $\zeta(\cdot)$ is the Riemann Zeta function, and the closed-form value for the $\overline{\bbH}^4$ case first appears in this paper.

We summarize the results of this paper in Theorems \ref{thm:6}--\ref{thm:9} as follows. 

\begin{theorem}
The optimal horoball packing density for noncompact Coxeter simplex tilings in
$\overline{\bbH}^6$ is $\delta_{opt} = \frac{81}{4 \sqrt{2}\pi^3}$, 
$\overline{\bbH}^7$ is $\delta_{opt} = \frac{28}{81\text{L}(4,3)}$,
$\overline{\bbH}^8$ is $\delta_{opt} = \frac{225}{8 \pi^4}$,
and in $\overline{\bbH}^9$ is $\delta_{opt} = \frac{1}{4 \zeta(5)}$, where 
$\text{L}(\cdot, \cdot)$ is the Dirichlet $L$-Series.
\label{thm:main}
\end{theorem}

Upper bounds for the packing density were published by Kellerhals \cite{K98} using the simplicial density function $d_n(\infty)$. This bound is strict for $n=3$, and Table \ref{table:summary} summarizes our main results where $\Delta$ is the gap between the packing density upper bound and our effective lower bounds, cf. Corollaries \ref{thm:6}--\ref{thm:9}. 

\begin{table}
    \begin{tabular}{l|l|l|l|l}
    \hline
    $n$ & Optimal Coxeter simplex packing density & Numerical Value & $d_n(\infty)$ & $\Delta$ \\
    \hline
    3 & $\left( 1+\frac{1}{2^2} - \frac{1}{4^2} - \frac{1}{5^2} + \frac{1}{7^2} + \frac{1}{8^2} - \dots \right)^{-1}$ & 0.85328\dots & 0.85328\dots & 0\\
    4 & $5 \sqrt{2} /\pi ^{2} $  & 0.71644\dots & 0.73046\dots & 0.0140\dots\\
    5 & $\displaystyle 5 / \left(7 \zeta(3)\right)$ & 0.59421\dots & 0.60695\dots & 0.0127\dots \\
    6 & $81/ \left(4 \sqrt{2} \pi ^{3} \right)$ & 0.46180\dots & 0.49339\dots & 0.0315\dots \\
    7 & $\displaystyle 28/ \left(81 \text{L}(4,3)\right)$ & 0.36773\dots & 0.39441\dots & 0.0266\dots\\
    8 & $\displaystyle 225/\left(8\pi^4\right)$ & 0.288731\dots & 0.31114\dots & 0.0223\dots\\
    9 & $1/\left( 4 \zeta(5) \right) $ & 0.24109\dots & 0.24285\dots & 0.0017\dots\\
    \hline
    \end{tabular}
    \caption{Packing density upper and lower bounds for $\mathbb{H}^n$.}
    \label{table:summary}
\end{table}

New to this paper, the notion of `horoball type' with respect to a fundamental domain is strengthened using isometry invariant Busemann functions. We use Busemann functions to parameterize horoballs centered at $\xi \in \partial \bbH^n$ with respect to a marked reference point $o \in \bbH^n$ (alternatively a reference horoball through $\xi$ and $o$)  in the model of $\overline{\bbH}^n$, see Section \ref{busemann}. This new point of view shows that the optimal packings are cannot be made equivalent by repartitioning, a nontrivial hyperbolic isometry, or some paradoxical construction, and clarifies our earlier results. 
Our method for computing densities in the projective Cayley--Klein model is largely similar to the earlier lower dimensional cases, see Section \ref{s:lemmas}, although the computations in coordinates is more involved. Hence the procedure was imporved to obtain exact closed-form expressions for packing densities in arithmetic lattices in Section \ref{s:densities}.

\section{Background}

Let $X$ denote a space of constant curvature, either the $n$-dimensional sphere $\mathbb{S}^n$, 
Euclidean space $\mathbb{E}^n$, or 
hyperbolic space $\mathbb{H}^n$ with $n \geq 2$. An important question of discrete geometry is to find the highest possible packing density in $X$ by congruent non-overlapping balls of a given radius \cite{G--K}. 
The definition of packing density is critical in hyperbolic space as shown by B\"or\"oczky \cite{B78}, for the standard paradoxical construction see \cite{G--K} or \cite{R06}. 
The most widely accepted notion of packing density considers the local densities of balls with respect to their Dirichlet--Voronoi cells (cf. \cite{B78} and \cite{K98}). In order to study horoball packings in $\overline{\mathbb{H}}^n$, we use an extended notion of such local density. 

Let $B$ be a horoball of packing $\cB$, and $P \in \overline{\mathbb{H}}^n$ an arbitrary point. 
Define $d(P,B)$ to be the shortest distance from point $P$ to the horosphere $S = \partial B$, where $d(P,B)\leq 0$  if $P \in B$. The Dirichlet--Voronoi cell $\cD(B,\cB)$ of horoball $B$ is the convex body
\begin{equation}
\cD(B,\cB) = \{ P \in \mathbb{H}^n | d(P,B) \le d(P,B'), ~ \forall B' \in \cB \}. \notag
\end{equation}
Both $B$ and $\cD$ have infinite volume, so the standard notion of local density is
modified. Let $Q \in \partial{\mathbb{H}}^n$ denote the ideal center of $B$, and take its boundary $S$ to be the one-point compactification of Euclidean $(n-1)$-space.
Let $B_C^{n-1}(r) \subset S$ be the Euclidean $(n-1)$-ball with center $C \in S \setminus \{Q\}$.
Then $Q$ and $B_C^{n-1}(r)$ determine a convex cone 
$\cC^n(r) = cone_Q\left(B_C^{n-1}(r)\right) \in \overline{\mathbb{H}}^n$ with
apex $Q$ consisting of all hyperbolic geodesics passing through $B_C^{n-1}(r)$ with limit point $Q$. The local density $\delta_n(B, \cB)$ of $B$ to $\cD$ is defined as
\begin{equation}
\delta_n(\cB, B) =\varlimsup\limits_{r \rightarrow \infty} \frac{vol(B \cap \cC^n(r))} {vol(\cD \cap \cC^n(r))}. \notag
\end{equation}
This limit is independent of the choice of center $C$ for $B^{n-1}_C(r)$.

In the case of periodic ball or horoball packings, this local density defined above extends to the entire hyperbolic space via its symmetry group, and 
is related to the simplicial density function (defined below) that we generalized in \cite{Sz12} and \cite{Sz12-2}.
In this paper we shall use such definition of packing density (cf. Section \ref{s:lemmas}).  

A Coxeter simplex is a top dimensional simplex in $X$ with dihedral angles either integral submultiples of $\pi$ or zero. 
The group generated by reflections on the sides of a Coxeter simplex is a Coxeter simplex reflection group. 
Such reflections generate a discrete group of isometries of $X$ with the Coxeter simplex as the fundamental domain; 
hence the groups give regular tessellations of $X$ if the fundamental simplex is characteristic. The Coxeter groups are finite for $\mathbb{S}^n$, and infinite for $\mathbb{E}^n$ or $\overline{\mathbb{H}}^n$. 

There are non-compact Coxeter simplices in $\overline{\mathbb{H}}^n$ with ideal vertices on $\partial \mathbb{H}^n$, however only for dimensions $2 \leq n \leq 9$; furthermore, only a finite number exist in dimensions $n \geq 3$. 
Johnson {\it et al.} \cite{JKRT} found the volumes of all Coxeter simplices in hyperbolic $n$-space. 
Such simplices are the most elementary building blocks of hyperbolic manifolds,
the volume of which is an important topological invariant. 

In $n$-dimensional space $X$ of constant curvature
 $(n \geq 2)$, define the simplicial density function $d_n(r)$ to be the density of $n+1$ mutually tangent balls of radius $r$ in the simplex spanned by their centers. L.~Fejes T\'oth and H.~S.~M.~Coxeter
conjectured that the packing density of balls of radius $r$ in $X$ cannot exceed $d_n(r)$.
Rogers \cite{Ro64} proved this conjecture in Euclidean space $\mathbb{E}^n$.
The $2$-dimensional spherical case was settled by L.~Fejes T\'oth \cite{FTL}, and B\"or\"oczky \cite{B78} gave a proof for the extension:
\begin{theorem}[K.~B\"or\"oczky]
In an $n$-dimensional space of constant curvature, consider a packing of spheres of radius $r$.
In the case of spherical space, assume that $r<\frac{\pi}{4}$.
Then the density of each sphere in its Dirichlet--Voronoi cell cannot exceed the density of $n+1$ spheres of radius $r$ mutually
touching one another with respect to the simplex spanned by their centers.
\label{t:Boroczky}
\end{theorem}

In hyperbolic 3-space, 
the monotonicity of $d_3(r)$ was proved by B\"or\"oczky and Florian
in \cite{B--F64}; in \cite{Ma99} Marshall 
showed that for sufficiently large $n$, 
function $d_n(r)$ is strictly increasing in variable $r$. Kellerhals \cite{K98} showed $d_n(r)<d_{n-1}(r)$, and that in cases considered by Marshall the local density of each ball in its Dirichlet--Voronoi cell is bounded above by the simplicial horoball density $d_n(\infty)$. Theorem \ref{t:Boroczky} is extended to the horoball case in \cite[\S 6]{B78} as a remark.

The simplicial packing density upper bound $d_3(\infty) = (1+\frac{1}{2^2}-\frac{1}{4^2}-\frac{1}{5^2}+\frac{1}{7^2}+\frac{1}{8^2}--++\dots)^{-1} = 0.85327\dots$ cannot be achieved by packing regular balls, instead it is realized by horoball packings of
$\overline{\mathbb{H}}^3$, the regular ideal simplex tiles $\overline{\mathbb{H}}^3$.
More precisely, the centers of horoballs in  $\partial\overline{\mathbb{H}}^3$ lie at the vertices of the ideal regular Coxeter simplex tiling with Schl\"afli symbol $[3,3,6]$. 

In three dimensions the B\"or\"oczky-type bound for horoball packings are used for volume estimates of cusped hyperbolic manifolds \cite{A87,M86}, more recently \cite{ACS,MaM}. Lifts of horoball neighborhoods of cusps give horoball packings in the universal cover $\bbH^n$, and for some discrete torsion free subgroup of isometries $\mathbb{H}^n/\Gamma$ is a cusped hyperbolic manifold where the cusps lift to ideal vertices of the fundamental domain. In this setting a manifold with a single cusp has a well defined maximal cusp neighborhood, while manifolds with multiple cusps have a range of non-overlapping cusp neighborhoods with boundaries with nonempty tangential intersection, these lift to different horoball types in the universal cover. An important application is Adams' proof that the Geiseking manifold is the noncompact hyperbolic $3$-manifold of minimal volume \cite{A87}. Kellerhals then used the B\"or\"oczky-type bounds to estimate volumes of higher dimensional hyperbolic manifolds \cite{K98_2}.  

In \cite{KSz} we proved that the classical horoball packing configuration in $\mathbb{H}^3$ realizing the B\"or\"oczky-type upper bound  is not unique. We gave several examples of different regular horoball packings using horoballs of different types, that is horoballs that have different relative densities with respect to the fundamental domain, that yield the B\"or\"oczky--Florian-type simplicial upper bound \cite{B--F64}.
 
Furthermore, in \cite{Sz12,Sz12-2} we found that 
by allowing horoballs of different types at each vertex of a totally asymptotic simplex and generalizing 
the simplicial density function to $\overline{\mathbb{H}}^n$ for $n \ge 2$,
 the B\"or\"oczky-type density 
upper bound is not valid for the fully asymptotic simplices for $n \geq 4$. 
In $\overline{\mathbb{H}}^4$ the locally optimal simplicial packing density is $0.77038\dots$, higher than the B\"or\"oczky-type density upper bound of $d_4(\infty) = 0.73046\dots$ using horoballs of a single type. 
However these ball packing configurations are only locally optimal and cannot be extended to the entirety of $\overline{\mathbb{H}}^n$. Finally, we mention the second-named author's preliminary results on horoball packings that motivated our collaboration \cite{Sz05-2,Sz07-1}.

\section{Preliminaries}

We use the projective Cayley--Klein model of hyperbolic geometry to preserves lines and convexity for the packing of simplex tilings with convex fundamental domains.  Hyperbolic symmetries are modeled as Euclidean projective transformations using the projective linear group $\text{PGL}(n+1,\mathbb{R})$.
In this section we review some key concepts, for a general discussion of the projective models of Thurston geometries see \cite{Mol97,MSz}.

\subsection{The Projective Model of  $\overline{\mathbb{H}}^n$}

Let $\mathbb{E}^{1,n}$ denote $\mathbb{R}^{n+1}$ with the Lorentzian inner product 
$\langle \mathbf{x}, \mathbf{y} \rangle = -x^0y^0+x^1y^1+ \dots + x^n y^n \label{bilinear_form}$
where non-zero real vectors 
$\mathbf{x, y}\in\mathbb{R}^{n+1}$ represent points in projective space 
$\mathbb{P}^n=\mathbb{P}(\mathbb{E}^{n+1})$, equipped with the quotient topology of the natural projection $\Pi: \mathbb{E}^{n+1}\setminus \{\mathbf{0}\} \rightarrow \mathbb{P}^n$. Partitioning $\mathbb{E}^{1,n}$ into $Q_+=\{\mathbf v \in \mathbb{R}^{n+1}|\langle \mathbf v, \mathbf v \rangle >0\}$, $Q_0=\{\mathbf v|\langle \mathbf v, \mathbf v \rangle =0\}$, and $Q_-=\{\mathbf v|\langle \mathbf v, \mathbf v \rangle <0\}$, the proper points of hyperbolic $n$-space are $\mathbb{H}^n = \Pi(Q_-)$, $\partial \mathbb{H}^n = \Pi(Q_0)$ are the boundary or ideal points, we will refer to points in $\Pi(Q_+)$ as outer points, and $ \overline{\mathbb{H}}^n = \mathbb{H}^n \cup \partial \mathbb{H}^n$ as extended hyperbolic space.  

 Points $[\mathbf{x}], [\mathbf{y}] \in \mathbb{P}^n$ are conjugate when $\langle
\mathbf{x}, \mathbf{y} \rangle = 0$. The set of all points conjugate
to $[\mathbf{x}]$ form a projective (polar) hyperplane
$pol([\mathbf{x}])=\{[\mathbf{y}] \in\mathbb{P}^n | \langle  \mathbf{x}, \mathbf{y} \rangle =0 \}.$
Hence $Q_0$ induces a duality $\mathbb{R}^{n+1} \leftrightarrow
\mathbb{R}_{n+1}$
between the points and hyperplanes of $\mathbb{P}^n$.
Point $[\mathbf{x}]$ and hyperplane $[\Ba]$ are incident if the value of
the linear form $\Ba$ evaluated on vector $\mathbf{x}$ is
 zero, i.e. $\mathbf{x}\Ba=0$ where $\mathbf{x} \in \
\mathbb{R}^{n+1} \setminus \{\mathbf{0}\}$, and $\Ba \in
\mathbb{R}_{n
+1} \setminus \{\mathbf{0}\}$.
Similarly, the lines in $\mathbb{P}^n$ are given by
2-subspaces of $\mathbb{R}^{n+1}$ or dual $(n-1)$-subspaces of $\mathbb{R}_{n+1}$ \cite{Mol97}.

Let $P \subset \overline{\mathbb{H}}^n$ be a polyhedron bounded by
a finite set of hyperplanes $H^i$ with unit normals
$\Bb^i \in \mathbb{R}_{n+1}$ directed
 towards the interior of $P$:
\begin{equation}
H^i=\{\mathbf{x} \in \mathbb{H}^n | \mathbf{x} \Bb^i =0 \} \ \ \text{with} \ \
\langle \Bb^i,\Bb^i \rangle = 1.
\end{equation}
In this paper $P$ is assumed to be an acute-angled polyhedron
 with proper or ideal vertices.
The Gram matrix of $P$ is $G(P)=( \langle \Bb^i,
\Bb^j \rangle )_{i,j}$, ${i,j \in \{ 0,1,2, \dots, n \} }$  is
 symmetric with signature $(1,n)$, its entries satisfy 
$\langle \Bb^i,\Bb^i \rangle = 1$
and $\langle \Bb^i,\Bb^j \rangle
\leq 0$ for $i \ne j$ where

$$
\langle \mbox{\boldmath$b$}^i,\mbox{\boldmath$b$}^j \rangle =
\left\{
\begin{aligned}
&0 & &\text{if}~H^i \perp H^j,\\
&-\cos{\alpha^{ij}} & &\text{if}~H^i,H^j ~ \text{intersect \ along an edge of $P$ \ at \ angle} \ \alpha^{ij}, \\
&-1 & &\text{if}~\ H^i,H^j ~ \text{are parallel in the hyperbolic sense}, \\
&-\cosh{l^{ij}} & &\text{if}~H^i,H^j ~ \text{admit a common perpendicular of length} \ l^{ij}.
\end{aligned}
\right.
$$
This is summarized in the Coxeter graph of the polytope $\sum(P)$. The graph nodes correspond 
to the hyperplanes $H^i$ and are connected if $H^i$ and $H^j$ are not perpendicular ($i \neq j$).
If connected the positive weight $k$ where  $\alpha_{ij} = \pi / k$ is indicated on the edge,  
unlabeled edges denote an angle of $\pi/3$.  Coxeter diagrams appear in Table \ref{table:simplex_list}.

In this paper we set the sectional curvature of $\mathbb{H}^n$,
$K=-k^2$, to be $k=1$. The distance $d$ between two proper points
$[\mathbf{x}]$ and $[\mathbf{y}]$ is given by
\begin{equation}
\cosh{{d}}=\frac{-\langle ~ \mathbf{x},~\mathbf{y} \rangle }{\sqrt{\langle ~ \mathbf{x},~\mathbf{x} \rangle
\langle ~ \mathbf{y},~\mathbf{y} \rangle }}.
\label{prop_dist}
\end{equation}
The projection $[\mathbf{y}]$ of point $[\mathbf{x}]$ on plane $[\Bu]$ is given by
\begin{equation}
\mathbf{y} = \mathbf{x} - \frac{ \langle \mathbf{x}, \mathbf{u} \rangle }{\langle \Bu, \Bu \rangle} \mathbf{u},
\label{perp_foot}
\end{equation}
where $[\mathbf u]$ is the pole of the plane $[\Bu]$.

\subsection{Horospheres and Horoballs in $\overline{\mathbb{H}}^n$}

A horosphere in $\overline{\mathbb{H}}^n$ ($n \ge 2)$ is as 
hyperbolic $n$-sphere with infinite radius centered 
at an ideal point $\xi \in \partial \mathbb{H}^n$ obtained as a limit of spheres through $x \in \bbH^n$ as its center $c \rightarrow \xi$. Equivalently, a horosphere is an $(n-1)$-surface orthogonal to
the set of parallel straight lines passing through $\xi \in \partial \mathbb{H}^n$. 
A horoball is a horosphere together with its interior. 

To derive the equation of a horosphere, fix a projective 
coordinate system for $\mathbb{P}^n$ with standard basis 
$\bold{a}_i, 0 \leq i \leq n $ so that the Cayley--Klein ball model of $\overline{\mathbb{H}}^n$ 
is centered at $O = (1,0,0,\dots, 0)$, and orient it by setting point $\xi \in \partial \bbH^n$ to lie at $A_0=(1,0,\dots, 0,1)$. 
The equation of a horosphere with center
$\xi = A_0$ passing through interior point $S=(1,0,\dots,0,s)$ is derived from the equation of the 
the boundary sphere $-x^0 x^0 +x^1 x^1+x^2 x^2+\dots + x^n x^n = 0$, and the plane $x^0-x^n=0$ tangent to the boundary sphere at $\xi = A_0$. 
The general equation of the horosphere is
\begin{equation}
0=\lambda (-x^0 x^0 +x^1 x^1+x^2 x^2+\dots + x^n x^n)+\mu{(x^0-x^n)}^2.
\label{horopshere_eq}
\end{equation}
Evaluating at $S$ obtain
\begin{equation}
\lambda (-1+s^2)+\mu {(-1+s)}^2=0 \text{~~and~~} \frac{\lambda}{\mu}=\frac{1-s}{1+s}. \notag
\end{equation}
For $s \neq \pm1$, the equation of a horosphere in projective coordinates is
\begin{align}
\label{eq:horosphere}
(s-1)\left(-x^0 x^0 +\sum_{i=1}^n (x^i)^2\right)-(1+s){(x^0-x^n)}^2 & =0.
\end{align}

In $\overline{\mathbb{H}}^n$ there exists an isometry $g \in \text{Isom}(\bbH^n)$ for any two horoballs $B$ and $B'$  such that $g.B = B'$.
However, it is often useful to distinguish between certain horoballs of a packing; we shall use the notion of horoball type with respect to the fundamental domain of a tiling (lattice) as introduced in \cite{Sz12-2}. In Section \ref{busemann} we show that this coincides with the Busemann function up to scaling, hence is isometry invariant. 

Two horoballs of a horoball packing are said to be of the {\it same type} or {\it equipacked} if 
and only if their local packing densities with respect to a particular cell (in our case a 
Coxeter simplex) are equal, otherwise the two horoballs are of {\it different type}. 
For example, the horoballs centered at $A_0$ passing through $S$ with different values for the final coordinate $s \in (-1,1)$ are of different type relative to 
a given cell, and the set of all horoballs centered at vertex $A_0$ is a one-parameter family.

Volumes of horoball pieces are given by J\'anos Bolyai's classical formulas from the mid 19-th century.
The hyperbolic length $L(x)$ of a horospherical arc contained in a chord segment of length $x$ is
\begin{equation}
\label{eq:horo_dist}
L(x)=2 \sinh{\left(\tfrac{x}{2}\right)} .
\end{equation}
The intrinsic geometry of a horosphere is Euclidean, 
so the $(n-1)$-dimensional volume $\mathcal{A}$ of a polyhedron $A$ on the 
surface of the horosphere can be calculated as in $\mathbb{E}^{n-1}$.
The volume of the horoball piece $\mathcal{H}(A)$ bounded by $A$, 
the set consisting of the union of geodesic segments
joining $A$ to the center of the horoball, is
\begin{equation}
\label{eq:bolyai}
vol(\mathcal{H}(A)) = \tfrac{1}{n-1}\mathcal{A}.
\end{equation}

\subsection{The Busemann function in $\overline{\bbH}^n$}
\label{busemann}

Define the Busemann function on $\overline{\bbH}^n$ as the map
$\beta: \bbH^n \times \bbH^n \times  \partial \bbH^n \rightarrow \bbR$ with $\beta(x,y,\xi)=\lim_{z \rightarrow \xi} \left(d(x,z) - d(y,z)\right)$, where the limit $z \rightarrow \xi$ is taken along any geodesic in $\bbH^n$ ending at boundary point $\xi$.
The Busemann function  satisfies $\beta(x, x, \xi) = 0$,  antisymmetry $\beta(x, y,\xi) = - \beta(y, x,\xi)$, the cocycle property $\beta(x, y,\xi) + \beta( y, z,\xi) = \beta(x, z,\xi)$ for all $x,y,z \in \bbH^n$, and is invariant under actions of $\text{Isom}(\bbH^n)$.
A horosphere centered at $\xi$ through $o$ is the level set of the Busemann function 
$\Hor_{\xi}(o) = \{x \in \bbH^n | \beta(x, o,\xi) = 0 \}$, while a horoball is the sublevel set $\Hor_{\xi}(o) = \{x \in \bbH^n | \beta(x, o,\xi) \le 0 \}$.
The space of all horospheres $\Hor(\bbH^n)$  gives an $\bbR$-fibration 
$h: \Hor(\bbH^n) \rightarrow \partial \bbH^n$ where $\Hor_{\xi}(o) \mapsto \xi$. 
The Busemann function then is an oriented distance between two concentric horospheres $\Hor_{\xi}(o_1)$ and $\Hor_{\xi}(o_2)$. For Busemann functions in Hadamard spaces defined by various authors cf. \cite{burger2013rigidity}, we adopt \cite{K02}.

Set reference point $o \in \bbH^n$ for the model at $o = (1, 0, \dots, 0)$ and reference horosphere $\Hor_{\xi}(o)$ at $\xi = (1, 0, \dots, 0, 1)$. The $s$-parameter of horosphere $\Hor_{\xi}(x)$ is $s = \text{th}( \beta(o, x,\xi))$ where $\text{th}(\cdot)$ is the hyperbolic tangent function.
A choice of reference point $o \in \bbH^n$ gives a trivialization of the fibration 
according to diagram 
\begin{center}
\begin{tikzcd}
\Hor(\bbH^n) \arrow{r}{\varphi_o} \arrow{d}[swap]{h} & \partial \bbH^n \times \bbR \arrow{ld}{\pi}  \\
 \partial \bbH^n
\end{tikzcd}
\end{center}
where 
$\Hor_{\xi}(x) \mapsto (\xi, \beta(o,x,\xi))$.
An element $g \in \text{Isom}(\bbH^n)$ acts on a horosphere as an additive cocycle
\begin{align*}
g.\Hor_{\xi}(x) = \Hor_{g.\xi}(gx) \mapsto & \left( g\xi, \beta(o, gx, g\xi)\right) = \left( g\xi, \beta(g^{-1}o, x, \xi)\right)\\
& = \left( g\xi, \beta(o, x,\xi) + \beta(g^{-1}o, o, \xi)\right).
\end{align*}
 Let $\hat s = \text{arcth}(s)$ then $g$ acts on the trivialization by
 \[
 g(\xi,\hat s) = (g\xi, \hat s+\beta(g^{-1}o,o,\xi).
 \]

In summary Busemann functions are related to the $s$-parameters by scaling and describe packing configurations relative to a marked point $o$ in an isometry invariantly.

\section{Packing Density in the Projective Model}
\label{s:lemmas}

In this bsection we define packing density and collect three Lemmas used in Section \ref{s:densities} to find the optimal packing densities for the Koszul simplex tilings. 

Let $\cT$ be a Coxeter tiling of $\overline{\mathbb{H}}^n$ \cite{JKRT2}. The symmetry group of a Coxeter tiling $\Gamma_\cT$ 
 contains its Coxeter group, and isometric mapping between two cells in $\cT$ preserves the tiling.
Any simplex cell of $\cT$ acts as a fundamental domain $\cF_{\cT}$
of $\Gamma_\cT$, and the Coxeter group is generated by reflections on the $(n - 1)$-dimensional facets of $\cF_{\cT}$. 
In this paper we consider only noncompact or Koszul-type Coxeter simplices, that is simplices with one or more ideal vertex, then the orbifold $\mathbb{H}^n/\Gamma_{\cT}$ has at least one cusp. In Table \ref{table:simplex_list} we list the 14 Koszul-type Coxeter simplices in $\overline{\mathbb{H}}^n$ for $6 \le n \le 9$, and their volumes. For a detailed discussion of the volume formulae for the other hyperbolic Coxeter simplices of dimensions $n \geq 3$, see Johnson {\it et al.} \cite{JKRT}. 

Define the density of a regular horoball packing $\mathcal{B}_{\cT}$ of Coxeter simplex tiling $\cT$ as
\begin{equation}
\delta(\mathcal{B}_{\cT})=\frac{\sum_{i=1}^m vol(B_i \cap \cF_{\cT})}{vol(\cF_{\cT})}.
\label{eq:density}
\end{equation}
$\cF_{\cT}$ denotes the simplex fundamental domain of tiling $\cT$, $m$ the number of ideal vertices of $\cF_\cT$, 
and $B_i$ the horoball centered at the $i$-th ideal vertex. 
We allow horoballs of different types at each asymptotic vertex of the tiling. 
A particular set of horoballs $\{B_i\}_{i=1}^m$ with different horoball types is allowed if it gives a packing: no two horoballs may have an interior point in common, and we require that no horoball extend beyond the facet opposite to the vertex where it is centered. The second condition ensures that the packing remains invariant under the actions of $\Gamma_\cT$ with $\cF_\cT$.
With these conditions satisfied, the packing density in $\cF_{\cT}$ extends to the entire $\overline{\mathbb{H}}^5$ by actions of $\Gamma_{\tau}$.
In the case of Coxeter simplex tilings, Dirichlet--Voronoi cells coincide with the Coxeter simplices. We denote the optimal horoball packing density as
\begin{equation}
\delta_{opt}(\cT) = \sup\limits_{\mathcal{B}_{\cT} \text{~packing}} \delta(\mathcal{B}_{\tau}).
\label{eq:opt_density}
\end{equation}

Let $\cF_{\Gamma}$ denote the simplicial fundamental domain of Coxeter tiling $\cT_{\Gamma}$ with vertex set $\{A_i\}_{i=0}^n \in \mathbb{P}(E^{1,n})$, where $A_0=(1, 0, \dots, 0, 1)$ is ideal and $A_1 = (1, 0, \dots, 0)$ is the center of the model $O$. Vertex coordinates $A_2, \dots, A_n$ then are set according to the dihedral angles of $\cF_{\Gamma}$ indicated in the Coxeter diagrams in Table \ref{table:simplex_list}, see Tables \ref{table:data_6dim}--\ref{table:data_9dim} for a choice of vertices, here $\Bu_i$ denote the hyperplane opposite to vertex $A_i$.

Lemma \ref{lem:loc} describes a procedure for finding the optimal horoball packing density in the fundamental domain $\cF_{\Gamma}$ with a single ideal vertex $A_0$. 
Packing density is maximized by the largest horoball type admissible in cell $\cF_{\Gamma}$ centered at $A_0$. Let $\mathcal{B}_0(s)$ denote the 1-parameter family of horoballs centered at $A_0$ where 
$s$-parameter related to the Busemann function measures the ``radius" of the horoball, the minimal Euclidean signed distance between the horoball and the center of the model $O$, taken negative if the horoball contains the model center.

\begin{lemma} [Local horoball density]
\label{lem:loc}
The local optimal horoball packing density of simply asymptotic Coxeter simplex $\cF_{\Gamma}$ is $\delta_{opt}(\Gamma) = \frac{vol(\mathcal{B}_0 \cap \cF_{\Gamma})}{vol(\cF_{\Gamma})}$.
\end{lemma}

\begin{proof} The maximal horoball $\mathcal{B}_0(s)$ opposite $A_0$ with fundamental domain $\cF_{\Gamma}$ is tangent to the hyperface of the simplex given by $\Bu_0$. This tangent point of $\mathcal{B}_0(s)$ and hyperface $\Bu_0$ is $[\mathbf{f}_0]$ the projection of vertex $A_0$ on plane $\Bu_0$ given by,
\begin{equation}
\mathbf{f}_0 =\ba_0 - \frac{\langle \ba_0, \mathbf{u}_0 \rangle}{\langle \Bu_0,\Bu_0 \rangle} 
\mathbf{u}_0.
\label{eq:u4fp}
\end{equation}

The value of the $s$-parameter for the maximal horoball can be read from the equation of the horosphere through $A_0$ and $\mathbf{f}_0$. 
The intersections $[\bh_i]$ of horosphere $\partial \mathcal{B}_0$ and the edges of the simplex $\cF_{\Gamma}$ are found by parameterizing the edges $\bh_i(\lambda) = \lambda \ba_0 +\ba_i$ $(1 \leq i \leq 5)$ then finding their intersections with $\partial \mathcal{B}_0$. 
The volume of the horospherical $(n-1)$-simplex determines the volume of the horoball piece by equation \eqref{eq:bolyai}.
The data for the horospheric $(n-1)$-simplex is obtained by finding hyperbolic distances $l_{ij}$ via equation \eqref{prop_dist},
$l_{ij} = d(H_i, H_j)$ where $d(\bh_i,\bh_j)= \arccos\left(\frac{-\langle \bh_i, \bh_j 
\rangle}{\sqrt{\langle \bh_i, \bh_i \rangle \langle \bh_j, \bh_j \rangle}}\right)$.
Moreover, the horospherical distances $L_{ij}$ are found by formula \eqref{eq:horo_dist}.
The intrinsic geometry of a horosphere is Euclidean, so the 
Cayley-Menger determinant gives the volume $\mathcal{A}$ of the horospheric $(n-1)$-simplex $\mathcal{A}$,

\begin{equation}
\mathcal{A}^2 = \frac{1}{(n!)^22^n}
\begin{vmatrix}
 0 & 1 & 1 & 1 & \dots &1 \\
 1 & 0 & L_{1,2}^2 & L_{1,3}^2 & \dots & L_{1,n}^2 \\
 1 & L_{1,2}^2 & 0 & L_{2,3}^2 & \dots & L_{2,n}^2 \\
 \vdots & \vdots & \vdots & \ddots & \dots & \vdots \\
 1 & L_{1,n}^2 & L_{2,n}^2 &  \dots & L_{n-1,n}^2 & 0
 \end{vmatrix}.
 \end{equation}

The volume of the horoball piece contained in the fundamental simplex is

\begin{equation}
vol(\mathcal{B}_0 \cap \cF_{\Gamma}) = \frac{1}{n-1}\mathcal{A}.
\end{equation}

The locally optimal horoball packing density of Coxeter Simplex $\cF_{\Gamma}$ is

\begin{equation}
\delta_{opt}(\cF_{\Gamma}) = \frac{vol(\mathcal{B}_0 \cap \cF_{\Gamma})}{vol(\cF_{\Gamma})}.
\end{equation}

\end{proof}

\begin{lemma}
\label{lem:glob}
The optimal horoball packing density $\delta_{opt}(\Gamma)$ of tiling $\cT_{\Gamma}$ and the local horoball packings density $\delta_{opt}(\cF_{\Gamma})$ are equal.
\end{lemma}

\begin{proof}
The local construction the the proof of Lemma \ref{lem:loc} is preserved by the isometric actions of $g \in \Gamma$. The Coxeter group $\Gamma$ extends the optimal local horoball packing density from the fundamental domain $\cF_{\Gamma}$ to the entire tiling $\cT_{\Gamma}$ of $\overline{\mathbb{H}}^n$, that is
$\delta_{opt}(\Gamma) = \delta_{opt}(\cF_{\Gamma})  = \frac{vol(\mathcal{B}_0 \cap \cF_{\Gamma})}{vol(\cF_{\Gamma})}$.

\end{proof}

The volumes of two tangent horoball pieces 
centered at two distinct ideal vertices of the fundamental domain as the horoball type is continuously varied are related in the Lemma \ref{lemma:szirmai}.

In $\overline{\mathbb{H}}^n$ with $n \geq 2$ let $\tau_1$ and $\tau_2$ be two congruent
$n$-dimensional convex cones with vertices at $C_1, C_2 \in \partial \overline{\mathbb{H}}^n$ that share a common geodesic edge $\overline{C_1C_2}$.
Let $B_1(x)$ and $B_2(x)$ denote two horoballs centered at $C_1$ and
$C_2$ respectively, mutually tangent at
$I(x)\in {\overline{C_1C_2}}$. Define $I(0)$ as the point with $V(0) = 2 vol(B_1(0) \cap \tau_1) = 2 vol(B_2(0) \cap \tau_2)$ for the volumes of the horoball sectors.

\begin{lemma}[\cite{Sz12}]
\label{lemma:szirmai}
Let $x$ be the hyperbolic distance between $I(0)$ and $I(x)$,
then
\begin{equation}
\begin{split}
V(x) =& vol(B_1(x) \cap \tau_1) + vol(B_2(x) \cap \tau_2) \\ =& V(0)\frac{e^{(n-1)x}+e^{-(n-1)x}}{2}
= V(0)\cosh\left((n-1)x\right) 
\end{split}
\end{equation}
is strictly convex and strictly increasing as $x\rightarrow\pm\infty$.
\end{lemma}

\begin{proof}
See our paper \cite{Sz12} for a proof. 
\end{proof}

\section{The Optimal Packing densities of the Koszul simplex tilings}
\label{s:densities}

In this section we determine the optimal horoball packing densities of the fourteen Koszul type Coxeter simplex tilings in dimensions $n = 6, 7, 8, 9$. Table \ref{table:simplex_list} summarizes the data and optimal packing density of each tiling. Fig. \ref{fig:lattice_of_subgroups} gives the commensurability relations of the groups in each dimension. We shall use the Witt symbols to denote each possible $\Gamma$.

\begin{table}
\resizebox{\columnwidth}{!}{%
    \begin{tabular}{cc|c|c|c}
    \hline
    Coxeter & ~ & Witt & Simplex & Optimal \\
    Diagram  & Notation & Symbol & Volume & Packing Density\\
    \hline
        $n=6$ dimensions & ~ & ~ & ~ & ~\\
    \hline
    \begin{tikzpicture}
	
	\draw (0.05,0) -- (2.5,0);
	\draw (1,0) -- (1,.5);

	\draw (0,0) circle (.05);
	\draw[fill=black] (.5,0) circle (.05); 
	\draw[fill=black] (1,0) circle (.05);
	\draw[fill=black] (1.5,0) circle (.05);
	\draw[fill=black] (2,0) circle (.05);
	\draw[fill=black] (2.5,0) circle (.05);
	\draw[fill=black] (1,.5) circle (.05);
	
	\node at (2.25,0.175) {$4$};

\end{tikzpicture}
  & $[4,3^2,3^{2,1}]$ & $\overline{S}_6$ & $\pi^3/777600$ & $\displaystyle \frac{81}{4 \sqrt{2} \pi ^3} \approx 0.46180\dots$  \\

    \begin{tikzpicture}
	
	\draw (0,0) -- (1.95,0);
	\draw (1,0) -- (1,.5);
	\draw (.5,0) -- (.5,.5);

	\draw[fill=black] (0,0) circle (.05);
	\draw[fill=black] (.5,0) circle (.05); 
	\draw[fill=black] (1,0) circle (.05);
	\draw[fill=black] (1.5,0) circle (.05);
	\draw (2,0) circle (.05);
	\draw[fill=black] (.5,.5) circle (.05);
	\draw[fill=black] (1,.5) circle (.05);
	
\end{tikzpicture}

  & $[3^{1,1},3,3^{2,1}]$ & $\overline{Q}_6$ & $\pi^3/388800$ & ''  \\

    \begin{tikzpicture}
	
	\draw (0.05,0) -- (.5,0);
	\draw (.5,0) -- (1,0.25);
	\draw (.5,0) -- (1,-0.25);
	\draw (1,.25) -- (1.5,0.25);
	\draw (1,-.25) -- (1.5,-0.25);
	\draw (1.5,.25) -- (2,0);
	\draw (1.5,-.25) -- (2,0);
	
	\draw (0,0) circle (.05);
	\draw[fill=black] (.5,0) circle (.05); 
	\draw[fill=black] (1,.25) circle (.05);
	\draw[fill=black] (1,-.25) circle (.05);
	\draw[fill=black] (1.5,0.25) circle (.05);
	\draw[fill=black] (1.5,-0.25) circle (.05);
	\draw[fill=black] (2,0) circle (.05);
	
\end{tikzpicture}

  & $[3, 3 ^{[6]}]$ & $\overline{P}_6$ & $\displaystyle 13 \pi^3/1360800$ & $\frac{189 \sqrt{3}}{26 \pi ^3} \approx 0.40609\dots$  \\

\hline
    $n=7$ dimensions & ~ & ~ & ~ & ~\\
\hline
    \begin{tikzpicture}
	
	\draw (-.45,0) -- (1,0);
	\draw (1,0) -- (1.5,0.25);
	\draw (1,0) -- (1.5,-0.25);
	\draw (1.5,.25) -- (2,.25);
	\draw  (1.5,-.25) -- (2,-.25);

	\draw (-.5,0) circle (.05);	
	\draw[fill=black] (0,0) circle (.05);
	\draw[fill=black] (0.5,0) circle (.05);
	\draw[fill=black] (1,0) circle (.05);
	\draw[fill=black] (1.5,0.25) circle (.05);
	\draw[fill=black] (1.5,-0.25) circle (.05);
	\draw[fill=black] (2,-0.25) circle (.05);
	\draw[fill=black] (2,0.25) circle (.05);
		
\end{tikzpicture}  

& $[3^{2,2,2}]$ & $\overline{T}_7$ & $\sqrt{3} \text{L}(4,3)/860160$ & $\displaystyle \frac{28}{81 \text{L}(4,3)} \approx 0.36773\dots$ \\
  
    \begin{tikzpicture}
	
	\draw (0.05,0) -- (3,0);
	\draw (1,0) -- (1,.5);

	\draw (0,0) circle (.05);
	\draw[fill=black] (.5,0) circle (.05); 
	\draw[fill=black] (1,0) circle (.05);
	\draw[fill=black] (1.5,0) circle (.05);
	\draw[fill=black] (2,0) circle (.05);
	\draw[fill=black] (2.5,0) circle (.05);
	\draw[fill=black] (3,0) circle (.05);
	\draw[fill=black] (1,.5) circle (.05);
	
	\node at (2.75,0.175) {$4$};

\end{tikzpicture}
  & $[4,3^3,3^{2,1}]$ & $\overline{S}_7$ & $\text{L}(4)/362880$ & $\displaystyle \frac{21}{64 \text{L}(4)} \approx 0.331793\dots$  \\
  
    \begin{tikzpicture}
	
	\draw (0.05,0) -- (2.5,0);
	\draw (1,0) -- (1,.5);
	\draw (2,0) -- (2,.5);

	\draw (0,0) circle (.05);
	\draw[fill=black] (.5,0) circle (.05); 
	\draw[fill=black] (1,0) circle (.05);
	\draw[fill=black] (1.5,0) circle (.05);
	\draw[fill=black] (2,0) circle (.05);
	\draw[fill=black] (2.5,0) circle (.05);
	\draw[fill=black] (1,.5) circle (.05);
	\draw[fill=black] (2,.5) circle (.05);

	\end{tikzpicture}
  & $[3^{1,1},3^2,3^{2,1}]$ & $\overline{Q}_7$ & $\text{L}(4)/181440$ & ''  \\

    \begin{tikzpicture}
	
	\draw (0.05,0) -- (.5,0);
	\draw (.5,0) -- (1,0.25);
	\draw (.5,0) -- (1,-0.25);
	\draw (1,.25) -- (2,0.25);
	\draw (1,-.25) -- (2,-0.25);
	\draw (2,.25) -- (2,-.25);
	
	\draw (0,0) circle (.05);
	\draw[fill=black] (.5,0) circle (.05); 
	\draw[fill=black] (1,.25) circle (.05);
	\draw[fill=black] (1,-.25) circle (.05);
	\draw[fill=black] (1.5,0.25) circle (.05);
	\draw[fill=black] (1.5,-0.25) circle (.05);
	\draw[fill=black] (2,0.25) circle (.05);
	\draw[fill=black] (2,-0.25) circle (.05);
	
    \end{tikzpicture}

  & $[3,3^{[7]}]$ & $\overline{P}_7$ & $7^{5/2} \text{L}(4,7)/3317760$ & $\displaystyle \frac{96}{343 \text{L}(4,7)} \approx 0.26605\dots$  \\

\hline
    $n=8$ dimensions & ~ & ~ & ~ & ~\\
\hline
    \begin{tikzpicture}
	
	\draw (0,0) -- (3.45,0);
	\draw (1.5,0) -- (1.5,.5);

	\draw[fill=black] (0,0) circle (.05);
	\draw[fill=black] (.5,0) circle (.05); 
	\draw[fill=black] (1,0) circle (.05);
	\draw[fill=black] (1.5,0) circle (.05);
	\draw[fill=black] (2,0) circle (.05);
	\draw[fill=black] (2.5,0) circle (.05);
	\draw[fill=black] (3,0) circle (.05);
	\draw (3.5,0) circle (.05);
	\draw[fill=black] (1.5,.5) circle (.05);
	
\end{tikzpicture}

  & $[3^{4,3,1}]$ & $\overline{T}_8$ & $ \pi^4/4572288000$ & $\displaystyle \frac{225}{8 \pi ^4} \approx 0.28873\dots$  \\
  
  \begin{tikzpicture}
	
	\draw (0.05,0) -- (.5,0);
	\draw (.5,0) -- (1,0.25);
	\draw (.5,0) -- (1,-0.25);
	\draw (1,.25) -- (2,0.25);
	\draw (1,-.25) -- (2,-0.25);
	\draw (2,.25) -- (2.46,0.04);
	\draw (2,-.25) -- (2.46,-0.04);
	
	\draw (0,0) circle (.05);
	\draw[fill=black] (.5,0) circle (.05); 
	\draw[fill=black] (1,.25) circle (.05);
	\draw[fill=black] (1,-.25) circle (.05);
	\draw[fill=black] (1.5,0.25) circle (.05);
	\draw[fill=black] (1.5,-0.25) circle (.05);
	\draw[fill=black] (2,0.25) circle (.05);
	\draw[fill=black] (2,-0.25) circle (.05);
	\draw (2.5,0) circle (.05);
	
    \end{tikzpicture}
    
  & $[3,3^{[8]}]$ & $\overline{P}_8$ & $17 \pi^3/285768000$ & ''  \\
  
    \begin{tikzpicture}
	
	\draw (0.05,0) -- (3.5,0);
	\draw (1,0) -- (1,.5);

	\draw (0,0) circle (.05);
	\draw[fill=black] (.5,0) circle (.05); 
	\draw[fill=black] (1,0) circle (.05);
	\draw[fill=black] (1.5,0) circle (.05);
	\draw[fill=black] (2,0) circle (.05);
	\draw[fill=black] (2.5,0) circle (.05);
	\draw[fill=black] (3,0) circle (.05);
	\draw[fill=black] (3.5,0) circle (.05);
	\draw[fill=black] (1,.5) circle (.05);
	
	\node at (3.25,0.175) {$4$};

\end{tikzpicture}
 
  & $[4,3^4,3^{2,1}]$ & $\overline{S}_8$ & $17 \pi^4/9144576000$ & $\displaystyle \frac{2025}{68 \sqrt{2} \pi ^4} \approx 0.21617\dots$  \\

    \begin{tikzpicture}
	
	\draw (0.05,0) -- (3,0);
	\draw (1,0) -- (1,.5);
	\draw (2.5,0) -- (2.5,.5);

	\draw (0,0) circle (.05);
	\draw[fill=black] (.5,0) circle (.05); 
	\draw[fill=black] (1,0) circle (.05);
	\draw[fill=black] (1.5,0) circle (.05);
	\draw[fill=black] (2,0) circle (.05);
	\draw[fill=black] (2.5,0) circle (.05);
	\draw[fill=black] (1,.5) circle (.05);
	\draw[fill=black] (2.5,.5) circle (.05);
	\draw[fill=black] (3,0) circle (.05);
	
	\end{tikzpicture}
	
  & $[4,3,3^{1,1,1}]$ & $\overline{Q}_8$ & $17 \pi^4/4572288000$ & ''  \\
\hline
    $n=9$ dimensions & ~ & ~ & ~ & ~\\
\hline

      \begin{tikzpicture}
	
	\draw (0.05,0) -- (3.95,0);
	\draw (1,0) -- (1,.5);

	\draw (0,0) circle (.05);
	\draw[fill=black] (.5,0) circle (.05); 
	\draw[fill=black] (1,0) circle (.05);
	\draw[fill=black] (1.5,0) circle (.05);
	\draw[fill=black] (2,0) circle (.05);
	\draw[fill=black] (2.5,0) circle (.05);
	\draw[fill=black] (3,0) circle (.05);
	\draw[fill=black] (3.5,0) circle (.05);
	\draw (4,0) circle (.05);
	\draw[fill=black] (1,.5) circle (.05);
	
	\node at (3.75,0.175) {$4$};

\end{tikzpicture}

  & $[4,3^5,3^{2,1}]$ & $\overline{S}_9$ & $ 527\zeta(5)/44590694400$ & $\displaystyle \frac{151}{1054 \zeta (5)} \approx 0.13816\dots$  \\
  
    \begin{tikzpicture}
	
	\draw (0,0) -- (3.95,0);
	\draw (1,0) -- (1,.5);

	\draw[fill=black] (0,0) circle (.05);
	\draw[fill=black] (.5,0) circle (.05); 
	\draw[fill=black] (1,0) circle (.05);
	\draw[fill=black] (1.5,0) circle (.05);
	\draw[fill=black] (2,0) circle (.05);
	\draw[fill=black] (2.5,0) circle (.05);
	\draw[fill=black] (3,0) circle (.05);
	\draw[fill=black] (3.5,0) circle (.05);
	\draw (4,0) circle (.05);
	\draw[fill=black] (1,.5) circle (.05);
	
\end{tikzpicture}

  & $[3^{6,2,1}]$ & $\overline{T}_9$ & $ \zeta(5)/222953472000$ & $\displaystyle \frac{1}{4 \zeta (5)}\approx 0.24109\dots$  \\

        \begin{tikzpicture}
	
	\draw (0.05,0) -- (3.45,0);
	\draw (1,0) -- (1,.45);
	\draw (3,0) -- (3,.45);

	\draw (0,0) circle (.05);
	\draw[fill=black] (.5,0) circle (.05); 
	\draw[fill=black] (1,0) circle (.05);
	\draw[fill=black] (1.5,0) circle (.05);
	\draw[fill=black] (2,0) circle (.05);
	\draw[fill=black] (2.5,0) circle (.05);
	\draw[fill=black] (3,0) circle (.05);
	\draw (3.5,0) circle (.05);
	\draw [fill=black] (1,.5) circle (.05);
	\draw (3,.5) circle (.05);

\end{tikzpicture}

  & $[3^{1,1},3^4,3^{2,1}]$ & $\overline{Q}_9$ & $ 527\zeta(5)/222953472000$ & ''  \\
  
\hline

    \end{tabular}%
    }
    \caption{Notation and volumes for the 14 asymptotic Coxeter Simplices in $\mathbb{H}^n$ for $6\le n \le 9$, empty circles in the Coxeter diagram denote reflection planes opposite an ideal vertex.}
    \label{table:simplex_list}
\end{table}

\begin{figure}
\begindc{\commdiag}[200]

\obj(0,8)[ss6]{$\overline{S}_6^{(1)}$}
\obj(0,6)[qq6]{$\overline{Q}_6^{(1)}$}
\obj(2,7)[pp6]{$\overline{P}_6^{(1)}$}
\mor{ss6}{qq6}{$2$}[\atright, \solidline]

\obj(4,8)[ss7]{$\overline{S}_7^{(1)}$}
\obj(4,6)[qq7]{$\overline{Q}_7^{(1)}$}
\obj(6,7)[tt7]{$\overline{T}_7^{(1)}$}
\obj(6,5)[pp7]{$\overline{P}_7^{(1)}$}
\mor{ss7}{qq7}{$2$}[\atright, \solidline]

\obj(9,8)[ss8]{$\overline{S}_8^{(1)}$}
\obj(9,6)[qq8]{$\overline{Q}_8^{(1)}$}
\obj(11,8)[tt8]{$\overline{T}_8^{(1)}$}
\obj(11,6)[pp8]{$\overline{P}_8^{(2)}$}
\mor{ss8}{qq8}{$2$}[\atright, \solidline]
\mor{tt8}{pp8}{$272$}[\atright, \solidline]

\obj(14,8)[ss9]{$\overline{S}_9^{(2)}$}
\obj(15,6)[qq9]{$\overline{Q}_9^{(3)}$}
\obj(16,8)[tt9]{$\overline{T}_9^{(1)}$}
\mor{ss9}{qq9}{$2$}[\atright, \solidline]
\mor{tt9}{qq9}{$527$}[\atleft, \solidline]

\enddc
\caption{Lattice of subgroups for each commensurability class of cocompact Coxeter groups. The subscript indicated the dimension, the superscript the number of ideal vertices of the fundamental simplex, and the index is indicated along edges.}
\label{fig:lattice_of_subgroups}
\end{figure}

\subsection{Case $n=6$ Dimensions}

\begin{theorem}
The optimal horoball packing density of Coxeter simplex tilings $\cT_\Gamma$, 
$\Gamma \in \Big\{ \overline{S}_6, \overline{Q}_6 \Big\}$
 is $\delta_{opt}(\Gamma) = \frac{81}{4 \sqrt{2}\pi^3}$, and for $\cT_{\overline{P}_6}$ is $\delta_{opt}(\overline{P}_6) = \frac{189\sqrt{3}}{26\pi^3}$.
\label{thm:6}
\end{theorem}

\begin{proof}
Each Coxeter simplex $\cF_{\Gamma}$ in $\overline{\mathbb{H}}^6$ has a single ideal vertex (see Table \ref{table:simplex_list}), so the local optimal packing densities follow from  Lemma \ref{lem:loc}, and extend to the entire space by Lemma \ref{lem:glob}.
Our choice of vertices $A_i$, forms of hyperplanes $\Bu_i$ opposite to vertices $A_i$, optimal horoball parameters $s$, and horoball intersection points are given in Table \ref{table:data_6dim}. 
\end{proof}

The following Corollary relates Theorem \ref{thm:6} to the simplicial packing density upper bound, recall Table \ref{table:summary}.

\begin{corollary}
The optimal congruent ball packing density in $\bbH^6$ up to horoballs of the same type is bounded by
$ \tfrac{81}{4 \sqrt{2}\pi^3} \leq \delta_{opt}(\overline{\mathbb{H}}^6) \leq 0.49339\dots.$
\end{corollary}

\begin{table}[h!]
\resizebox{\columnwidth}{!}{%
	\begin{tabular}{l|l|l|l}
		 \hline
		 \multicolumn{4}{c}{{\bf The 6 Dimensional Coxeter Simplex Tilings} }\\
		\hline
		 Witt Symb. & $\overline{S}_6$ &  $\overline{Q}_6$ &  $\overline{P}_6$  \\
		 \hline
		 \multicolumn{4}{c}{{\bf Vertices of Simplex} }\\
		 \hline
		 $A_0$ & $(1, 0, 0, 0, 0, 0, 1)$ & $(1, 0, 0, 0, 0, 0, 1)$ & $(1, 0, 0, 0, 0, 0, 1)$  \\
		 $A_1$ & $(1, 0, 0, 0, 0, 0, 0)$ & $(1, 0, 0, 0, 0, 0, 0)$ & $(1, 0, 0, 0, 0, 0, 0)$  \\
		 $A_2$ & $(1,0,0,0,0,\frac{1}{2},0)$ & $(1,0,0,0,0,\frac{1}{2},0)$ & $(1,0,0,0,0,\frac{\sqrt{15}}{6},0)$   \\
		 $A_3$ & $(1,0,0,0,\frac{\sqrt{2}}{4},\frac{1}{2},0)$  & $(1,0,0,0,\frac{\sqrt{2}}{4},\frac{1}{2},0)$ & $(1,0,0,0,\frac{\sqrt{10}}{5},\frac{2 \sqrt{15}}{15},0)$   \\
		 $A_4$ & $(1,0,0,\frac{\sqrt{2}}{4},\frac{\sqrt{2}}{4},\frac{1}{2},0)$ & $(1,0,0,\frac{1}{2},\frac{\sqrt{2}}{4},\frac{1}{2},0)$ & $(1,0,0,\frac{\sqrt{6}}{4},\frac{3 \sqrt{10}}{20},\frac{\sqrt{15}}{10},0)$  \\
		 $A_5$ & $(1,0,\frac{\sqrt{2}}{4},\frac{\sqrt{2}}{4},\frac{\sqrt{2}}{4},\frac{1}{2},0)$ & $(1,0,\frac{1}{2},0,\frac{\sqrt{2}}{4},\frac{1}{2},0)$ & $(1,0,-\frac{\sqrt{3}}{3},\frac{\sqrt{6}}{6},\frac{\sqrt{10}}{10},\frac{\sqrt{15}}{15},0)$  \\
		 $A_6$ & $(1,\frac{1}{2},0,0,0,\frac{1}{2},0)$ & $(1,-\frac{1}{2},0,0,0,\frac{1}{2},0)$ & $(1,-\frac{1}{2},-\frac{\sqrt{3}}{6},\frac{\sqrt{6}}{12},\frac{\sqrt{10}}{20},\frac{\sqrt{15}}{30},0)$  \\
		 \hline
		 \multicolumn{4}{c}{{\bf The form $\mbox{\boldmath$u$}_i$ of sides opposite $A_i$ }}\\
		\hline
		 $\mbox{\boldmath$u$}_0$ & $(0, 0, 0, 0, 0, 0, 1)$  & $(0, 0, 0, 0, 0, 0, 1)$ & $(0, 0, 0, 0, 0, 0, 1)$   \\
		 $\mbox{\boldmath$u$}_1$ & $(1, 0, 0, 0, 0, -2, -1)$ & $(1, 0, 0, 0, 0, -2, -1)$ & $(1,1,\frac{1}{\sqrt{3}},-\frac{1}{\sqrt{6}},-\frac{1}{\sqrt{10}},-2 \sqrt{\frac{3}{5}},-1)$   \\
		 $\mbox{\boldmath$u$}_2$ &  $(0,-1,0,0,-\sqrt{2},1,0)$ & $(0,1,0,0,-\sqrt{2},1,0)$  & $(0,0,0,0,-\sqrt{\frac{2}{3}},1,0)$  \\
		 $\mbox{\boldmath$u$}_3$ & $(0, 0, 0, -1, 1, 0, 0)$ & $(0,0,-\frac{1}{\sqrt{2}},-\frac{1}{\sqrt{2}},1,0,0)$ & $(0,0,0,-\sqrt{\frac{3}{5}},1,0,0)$  \\
		 $\mbox{\boldmath$u$}_4$ & $(0, 0, -1, 1, 0, 0, 0)$ & $(0, 0, 0, 1, 0, 0, 0)$ & $(0,0,\frac{1}{\sqrt{2}},1,0,0,0)$  \\
		 $\mbox{\boldmath$u$}_5$ & $(0, 0, 1, 0, 0, 0, 0)$ & $(0, 0, 1, 0, 0, 0, 0)$ & $(0,-\frac{\sqrt{3}}{3},1,0,0,0,0)$  \\
		 $\mbox{\boldmath$u$}_6$ & $(0, 1, 0, 0, 0, 0, 0)$ & $(0, 1, 0, 0, 0, 0, 0)$ & $(0, 1, 0, 0, 0, 0, 0)$  \\
		 \hline
		 \multicolumn{4}{c}{{\bf Maximal horoball parameter $s_0$ }}\\
		\hline
		 $s_0$ & $0$ & $0$ & $0$ \\
		\hline
		 \multicolumn{4}{c}{ {\bf Intersections $H_i = \mathcal{B}(A_0,s_0)\cap A_0A_i$ of horoballs with simplex edges}}\\
		\hline
		 $H_1$ & $(1, 0, 0, 0, 0, 0, 0)$ & $(1, 0, 0, 0, 0, 0, 0)$ & $(1, 0, 0, 0, 0, 0, 0)$   \\
		 $H_2$ & $(1,0,0,0,0,\frac{4}{9},\frac{1}{9})$ & $(1,0,0,0,0,\frac{4}{9},\frac{1}{9})$ & $(1,0,0,0,0,\frac{4 \sqrt{15}}{29},\frac{5}{29})$    \\
		 $H_3$ & $(1,0,0,0,\frac{4 \sqrt{2}}{19},\frac{8}{19},\frac{3}{19})$ & $(1,0,0,0,\frac{4 \sqrt{2}}{19},\frac{8}{19},\frac{3}{19})$ & $(1,0,0,0,\frac{3}{2 \sqrt{10}},\frac{\sqrt{\frac{3}{5}}}{2},\frac{1}{4})$    \\
		 $H_4$ & $(1,0,0,\frac{\sqrt{2}}{5},\frac{\sqrt{2}}{5},\frac{2}{5},\frac{1}{5})$ & $(1,0,0,\frac{8}{21},\frac{4 \sqrt{2}}{21},\frac{8}{21},\frac{5}{21})$ & $(1,0,0,\frac{2 \sqrt{6}}{11},\frac{6 \sqrt{\frac{2}{5}}}{11},\frac{4 \sqrt{\frac{3}{5}}}{11},\frac{3}{11})$    \\
		 $H_5$ & $(1,0,\frac{4 \sqrt{2}}{21},\frac{4 \sqrt{2}}{21},\frac{4 \sqrt{2}}{21},\frac{8}{21},\frac{5}{21})$ & $(1,0,\frac{8}{21},0,\frac{4 \sqrt{2}}{21},\frac{8}{21},\frac{5}{21})$ & $(1,0,-\frac{\sqrt{3}}{4},\frac{\sqrt{\frac{3}{2}}}{4},\frac{3}{4 \sqrt{10}},\frac{\sqrt{\frac{3}{5}}}{4},\frac{1}{4})$    \\
		 $H_6$ & $(1,\frac{2}{5},0,0,0,\frac{2}{5},\frac{1}{5})$ & $(1,-\frac{2}{5},0,0,0,\frac{2}{5},\frac{1}{5})$ & $(1,-\frac{12}{29},-\frac{4 \sqrt{3}}{29} ,\frac{2 \sqrt{6}}{29},\frac{6 \sqrt{\frac{2}{5}}}{29},\frac{4 \sqrt{\frac{3}{5}}}{29},\frac{5}{29})$    \\
		 \hline
		 \multicolumn{4}{c}{ {\bf Volume of maximal horoball piece }}\\
		\hline
		 $vol(\mathcal{B}_0 \cap \mathcal{F}_{\Gamma})$ & $(38400 \sqrt{2})^{-1}$ & $(19200 \sqrt{2})^{-1}$ & $(4800 \sqrt{3})^{-1}$ \\
		 \hline
		\multicolumn{4}{c}{ {\bf Optimal Packing Density}}\\
		\hline
		 $\delta_{opt}$ & $\frac{81}{4 \sqrt{2} \pi ^3} \approx 0.46180\dots$  & $\frac{81}{4 \sqrt{2} \pi ^3}$ & $\frac{189 \sqrt{3}}{26 \pi ^3} \approx 0.40606 \dots$  \\
		\hline
	\end{tabular}%
}
	\caption{Data for asymptotic Coxeter tilings of $\mathbb{H}^6$ in the Cayley-Klein ball model centered at $O=(1,0,0,0,0,0,0)$}
	\label{table:data_6dim}
\end{table}

\subsection{Case $n=7$ Dimensions}

\begin{theorem}
The optimal horoball packing density of Coxeter simplex tilings $\cT_\Gamma$, 
$\Gamma \in \Big\{ \overline{S}_7, \overline{Q}_7 \Big\}$
 is $\delta_{opt}(\Gamma) = \frac{21}{64 \text{L}(4)}$. 
The Coxeter simplex tiling $\cT_{\overline{P}_7}$ is $\delta_{opt}(\overline{P}_7) = \frac{96}{343\text{L}(4,7)}$, and $\cT_{\overline{T}_7}$ is $\delta_{opt}(\overline{T}_7) = \frac{28}{81\text{L}(4,3)}$.
\label{thm:7}
\end{theorem}

\begin{proof}
Each Coxeter simplex $\cF_{\Gamma}$ in $\overline{\mathbb{H}}^7$ has one ideal vertex (see Table \ref{table:simplex_list}), so the locally optimal packing densities follow from  Lemma \ref{lem:loc}, and extend to the entire space by Lemma \ref{lem:glob}.
Our choice of vertices $A_i$, forms of hyperplanes $\Bu_i$ opposite to vertices $A_i$, optimal horoball parameters $s$, and horoball intersection points are given in Table \ref{table:data_7dim}. Here we used the Dirichlet L-function $\text{L}(s,d)=\sum^{\infty}_{n=1}\left(\frac{n}{d}\right)n^{-s}$, where $(n/d)$ is the Legendre symbol. 
\end{proof}

\begin{corollary}
The optimal congruent ball packing density in $\bbH^7$ up to horoballs of the same type is bounded by
$ \tfrac{28}{81\text{L}(4,3)} \leq \delta_{opt}(\overline{\mathbb{H}}^7) \leq 0.39441\dots.$
\end{corollary}

\begin{landscape}

\begin{table}
\resizebox{\columnwidth}{!}{%
	\begin{tabular}{l|l|l|l|l}
		 \hline
		 \multicolumn{5}{c}{{\bf The 7 Dimensional Coxeter Simplex Tilings} }\\
		\hline
		 Witt Symb. & $\overline{S}_7$ &  $\overline{Q}_7$ &  $\overline{T}_7$ &  $\overline{P}_7$ \\
		 \hline
		 \multicolumn{5}{c}{{\bf Vertices of Simplex} }\\
		 \hline
		 $A_0$ & $(1, 0, 0, 0, 0, 0, 0, 1)$ & $(1, 0, 0, 0, 0, 0, 0, 1)$ & $(1, 0, 0, 0, 0, 0, 0, 1)$ & $(1, 0, 0, 0, 0, 0, 0, 1)$  \\
		 $A_1$ & $(1, 0, 0, 0, 0, 0, 0, 0)$ & $(1, 0, 0, 0, 0, 0, 0, 0)$ & $(1, 0, 0, 0, 0, 0, 0, 0)$ & $(1, 0, 0, 0, 0, 0, 0, 0)$  \\
		 $A_2$ & $(1,0,0,0,0,0,\frac{1}{2},0)$ & $(1,0,0,0,0,0,\frac{1}{2},0)$ & $(1,0,0,0,0,0,\frac{1}{2},0)$ & $(1,0,0,0,0,0,\frac{\sqrt{21}}{7},0)$  \\
		 $A_3$ & $(1,0,0,0,0,\frac{\sqrt{2}}{4},\frac{1}{2},0)$ & $(1,0,0,0,0,\frac{\sqrt{2}}{4},\frac{1}{2},0)$ & $(1,0,0,0,0,\frac{\sqrt{3}}{6},\frac{1}{2},0)$ & $(1,0,0,0,0,\frac{\sqrt{15}}{6},\frac{5 \sqrt{21}}{42},0)$  \\
		 $A_4$ & $(1,0,0,0,\frac{\sqrt{2}}{4},\frac{\sqrt{2}}{4},\frac{1}{2},0)$ & $(1,0,0,0,\frac{\sqrt{2}}{4},\frac{\sqrt{2}}{4},\frac{1}{2},0)$ & $(1,0,0,0,\frac{\sqrt{3}}{6},\frac{\sqrt{3}}{6},\frac{1}{2},0)$ & $(1,0,0,0,\frac{\sqrt{10}}{5},\frac{2 \sqrt{15}}{15},\frac{2 \sqrt{21}}{21},0)$  \\
		 $A_5$ & $(1,0,0,\frac{\sqrt{2}}{4},\frac{\sqrt{2}}{4},\frac{\sqrt{2}}{4},\frac{1}{2},0)$ & $(1,0,0,\frac{1}{2},\frac{\sqrt{2}}{4},\frac{\sqrt{2}}{4},\frac{1}{2},0)$ & $(1,0,0,\frac{1}{2},\frac{\sqrt{3}}{6},\frac{\sqrt{3}}{6},\frac{1}{2},0)$ & $(1,0,0,\frac{\sqrt{6}}{4},\frac{3 \sqrt{10}}{20},\frac{\sqrt{15}}{10},\frac{\sqrt{21}}{14},0)$  \\
		 $A_6$ & $(1,0,\frac{\sqrt{2}}{4},\frac{\sqrt{2}}{4},\frac{\sqrt{2}}{4},\frac{\sqrt{2}}{4},\frac{1}{2},0)$ & $(1,0,\frac{1}{2},0,\frac{\sqrt{2}}{4},\frac{\sqrt{2}}{4},\frac{1}{2},0)$ & $(1,0,-\frac{\sqrt{3}}{6},0,0,\frac{\sqrt{3}}{6},\frac{1}{2},0)$ & $(1,0,\frac{\sqrt{3}}{3},\frac{\sqrt{6}}{6},\frac{\sqrt{10}}{10},\frac{\sqrt{15}}{15},\frac{\sqrt{21}}{21},0)$  \\
		 $A_7$ & $(1,\frac{1}{2},0,0,0,0,\frac{1}{2},0)$ & $(1,\frac{1}{2},0,0,0,0,\frac{1}{2},0)$ & $(1,-\frac{1}{2},-\frac{\sqrt{3}}{6},0,0,\frac{\sqrt{3}}{6},\frac{1}{2},0)$ & $(1,\frac{1}{2},\frac{\sqrt{3}}{6},\frac{\sqrt{6}}{12},\frac{\sqrt{10}}{20},\frac{\sqrt{15}}{30},\frac{\sqrt{21}}{42},0)$  \\
		 \hline
		 \multicolumn{5}{c}{{\bf The form $\mbox{\boldmath$u$}_i$ of sides opposite $A_i$ }}\\
		\hline
		 $\mbox{\boldmath$u$}_0$ & $(0, 0, 0, 0, 0, 0, 0, 1)$ & $(0, 0, 0, 0, 0, 0, 0, 1)$ & $(0, 0, 0, 0, 0, 0, 0, 1)$ & $(0, 0, 0, 0, 0, 0, 0, 1)$  \\
		 $\mbox{\boldmath$u$}_1$ & $(1, 0, 0, 0, 0, 0, -2, -1)$ & $(1, 0, 0, 0, 0, 0, -2, -1)$ & $(1, 0, 0, 0, 0, 0, -2, -1)$ & $(1,-1,-\frac{1}{\sqrt{3}},-\frac{1}{\sqrt{6}},-\frac{1}{\sqrt{10}},-\frac{1}{\sqrt{15}},-\sqrt{\frac{7}{3}},-1)$  \\
		 $\mbox{\boldmath$u$}_2$ & $(0,-1,0,0,0,-\sqrt{2},1,0)$ & $(0,-1,0,0,0,-\sqrt{2},1,0)$ & $(0,0,0,0,0,-\sqrt{3},1,0)$ & $(0,0,0,0,0,-\sqrt{\frac{5}{7}},1,0)$  \\
		 $\mbox{\boldmath$u$}_3$ & $(0, 0, 0, 0, -1, 1, 0, 0)$ & $(0, 0, 0, 0, -1, 1, 0, 0)$ & $(0, 0, 1, 0, -1, 1, 0, 0)$ & $(0,0,0,0,-\sqrt{\frac{2}{3}},1,0,0)$  \\
		 $\mbox{\boldmath$u$}_4$ & $(0, 0, 0, -1, 1, 0, 0, 0)$ & $(0,0,-\frac{1}{\sqrt{2}},-\frac{1}{\sqrt{2}},1,0,0,0)$ & $(0,0,0,-\frac{1}{\sqrt{3}},1,0,0,0)$ & $(0,0,0,-\sqrt{\frac{3}{5}},1,0,0,0)$  \\
		 $\mbox{\boldmath$u$}_5$ & $(0, 0, -1, 1, 0, 0, 0, 0)$ & $(0, 0, 0, 1, 0, 0, 0, 0)$ & $(0, 0, 0, 1, 0, 0, 0, 0)$ & $(0,0,-\frac{1}{\sqrt{2}},1,0,0,0,0)$  \\
		 $\mbox{\boldmath$u$}_6$ & $(0, 0, 1, 0, 0, 0, 0, 0)$ & $(0, 0, 1, 0, 0, 0, 0, 0)$ & $(0,-\frac{1}{\sqrt{3}},1,0,0,0,0,0)$ & $(0,-\frac{1}{\sqrt{3}},1,0,0,0,0,0)$  \\
		 $\mbox{\boldmath$u$}_7$ & $(0, 1, 0, 0, 0, 0, 0, 0)$ & $(0, 1, 0, 0, 0, 0, 0, 0)$ & $(0, 1, 0, 0, 0, 0, 0, 0)$ & $(0, 1, 0, 0, 0, 0, 0, 0)$  \\
		 \hline
		 \multicolumn{5}{c}{{\bf Maximal horoball parameter $s_0$ }}\\
		\hline
		 $s_0$ & $0$ & $0$ & $0$ & $0$ \\
		\hline
		 \multicolumn{5}{c}{ {\bf Intersections $H_i = \mathcal{B}(A_0,s_0)\cap A_0A_i$ of horoballs with simplex edges}}\\
		\hline
		 $H_1$ & $(1, 0, 0, 0, 0, 0, 0, 0)$ & $(1, 0, 0, 0, 0, 0, 0, 0)$ & $(1, 0, 0, 0, 0, 0, 0, 0)$ & $(1, 0, 0, 0, 0, 0, 0, 0)$  \\
		 $H_2$ & $(1,0,0,0,0,0,\frac{4}{9},\frac{1}{9})$ & $(1,0,0,0,0,0,\frac{4}{9},\frac{1}{9})$ & $(1,0,0,0,0,0,\frac{4}{9},\frac{1}{9})$ & $(1,0,0,0,0,0,\frac{2 \sqrt{21}}{17},\frac{3}{17})$  \\
		 $H_3$ & $(1,0,0,0,0,\frac{4 \sqrt{2}}{19},\frac{8}{19},\frac{3}{19})$ & $(1,0,0,0,0,\frac{4 \sqrt{2}}{19},\frac{8}{19},\frac{3}{19})$ & $(1,0,0,0,0,\frac{\sqrt{3}}{7},\frac{3}{7},\frac{1}{7})$ & $(1,0,0,0,0,\frac{7 \sqrt{\frac{5}{3}}}{19},\frac{5 \sqrt{\frac{7}{3}}}{19},\frac{5}{19})$  \\
		 $H_4$ & $(1,0,0,0,\frac{\sqrt{2}}{5},\frac{\sqrt{2}}{5},\frac{2}{5},\frac{1}{5})$ & $(1,0,0,0,\frac{\sqrt{2}}{5},\frac{\sqrt{2}}{5},\frac{2}{5},\frac{1}{5})$ & $(1,0,0,0,\frac{4 \sqrt{3}}{29},\frac{4 \sqrt{3}}{29},\frac{12}{29},\frac{5}{29})$ & $(1,0,0,0,\frac{7}{5 \sqrt{10}},\frac{7}{5 \sqrt{15}},\frac{\sqrt{\frac{7}{3}}}{5},\frac{3}{10})$  \\
		 $H_5$ & $(1,0,0,\frac{4 \sqrt{2}}{21},\frac{4 \sqrt{2}}{21},\frac{4 \sqrt{2}}{21},\frac{8}{21},\frac{5}{21})$ & $(1,0,0,\frac{4}{11},\frac{2 \sqrt{2}}{11},\frac{2 \sqrt{2}}{11},\frac{4}{11},\frac{3}{11})$ & $(1,0,0,\frac{3}{8},\frac{\sqrt{3}}{8},\frac{\sqrt{3}}{8},\frac{3}{8},\frac{1}{4})$ & $(1,0,0,\frac{7 \sqrt{\frac{3}{2}}}{20},\frac{21}{20 \sqrt{10}},\frac{7 \sqrt{\frac{3}{5}}}{20},\frac{\sqrt{21}}{20},\frac{3}{10})$  \\
		 $H_6$ & $(1,0,\frac{2 \sqrt{2}}{11},\frac{2 \sqrt{2}}{11},\frac{2 \sqrt{2}}{11},\frac{2 \sqrt{2}}{11},\frac{4}{11},\frac{3}{11})$ & $(1,0,\frac{4}{11},0,\frac{2 \sqrt{2}}{11},\frac{2 \sqrt{2}}{11},\frac{4}{11},\frac{3}{11})$ & $(1,0,-\frac{4 \sqrt{3}}{29} ,0,0,\frac{4 \sqrt{3}}{29},\frac{12}{29},\frac{5}{29})$ & $(1,0,\frac{14}{19 \sqrt{3}},\frac{7 \sqrt{\frac{2}{3}}}{19},\frac{7 \sqrt{\frac{2}{5}}}{19},\frac{14}{19 \sqrt{15}},\frac{2 \sqrt{\frac{7}{3}}}{19},\frac{5}{19})$  \\
		 $H_7$ & $(1,\frac{2}{5},0,0,0,0,\frac{2}{5},\frac{1}{5})$ & $(1,\frac{2}{5},0,0,0,0,\frac{2}{5},\frac{1}{5})$ & $(1,-\frac{3}{8},-\frac{\sqrt{3}}{8},0,0,\frac{\sqrt{3}}{8},\frac{3}{8},\frac{1}{4})$ & $(1,\frac{7}{17},\frac{7}{17 \sqrt{3}},\frac{7}{17 \sqrt{6}},\frac{7}{17 \sqrt{10}},\frac{7}{17 \sqrt{15}},\frac{\sqrt{\frac{7}{3}}}{17},\frac{3}{17})$  \\
		 \hline
		 \multicolumn{5}{c}{ {\bf Volume of maximal horoball piece } }\\
		\hline
		 $vol(\mathcal{B}_0 \cap \mathcal{F})$ & $1105920^{-1}$ & $552960^{-1}$ & $(829440 \sqrt{3})^{-1}$ &  $34560 \sqrt{7})^{-1}$ \\		
		 \hline
		\multicolumn{5}{c}{ {\bf Optimal Packing Density} }\\
		\hline
		 $\delta_{opt}$ & $\frac{21}{64 \text{L}(4)} \approx 0.33179\dots$ & $\frac{21}{64 \text{L}(4)}$ & $\frac{28}{81 \text{L}(4,3)}0.36773\dots$ & $\frac{96}{343 \text{L}(4,7)} \approx 0.26605\dots$  \\
		\hline
	\end{tabular}%
}
	\caption{Data for asymptotic Coxeter tilings of $\mathbb{H}^7$ in the Cayley-Klein ball model centered at $O=(1,0,0,0,0,0,0,0)$}
	\label{table:data_7dim}
\end{table}

\end{landscape}

\subsection{Case $n=8$ Dimensions}
\begin{theorem}
The optimal horoball packing density of Coxeter simplex tilings $\cT_\Gamma$, 
$\Gamma \in \Big\{ \overline{S}_8, \overline{Q}_8 \Big\}$
 is $\delta_{opt}(\Gamma) = \frac{2025}{68 \sqrt{2} \pi^4}$, and for
 $\Gamma \in \Big\{ \overline{T}_8, \overline{P}_8 \Big\}$, $\delta_{opt}(\Gamma) = \frac{225}{8 \pi^4}$.
 
\label{thm:8}
\end{theorem}

\begin{proof}
There are two cases, the fundamental domain has one or two ideal vertices. 

Case 1: Coxeter simplices $\cF_{\Gamma}$ for $\Gamma \in \Big\{ \overline{S}_8, \overline{Q}_8, \overline{T}_8 \Big\}$ in $\overline{\mathbb{H}}^8$ have one ideal vertex and the local optimal packing densities follow from  Lemma \ref{lem:loc}, and extends to the entire space by Lemma \ref{lem:glob}.
Our choice of coordinates for vertices $A_i$, forms of hyperplanes $\Bu_i$ opposite to vertices $A_i$, and the computed optimal horoball $s$ parameters, horoball intersection points are given in Table \ref{table:data_8dim}.

Case 2: $\cF_{\overline{P}_8}$ has two ideal vertices 
$A_0$ and $A_5$, see   
Table \ref{table:data_8dim}. 
Let  $B_0\left( \arctanh s_0 \right)$ and $B_5(\arctanh s_5 )$ be horoballs
with parameters
$s_0$ and $s_5$ centered at $A_0$ and $A_5$. To find the horosphere equation for horoball $B_5$, we transform the model and rotate $A_5$ to $A_0$ by $\text{Rot}_{A_5 A_0} \in \text{PGL}(n+1,\bbR)$ in coordinates represented by matrix
\begin{equation}
\text{Rot}_{A_5A_0} = \left(
\begin{array}{ccccccccc}
 1 & 0 & 0 & 0 & 0 & 0 & 0 & 0 & 0 \\
 0 & 1 & 0 & 0 & 0 & 0 & 0 & 0 & 0 \\
 0 & 0 & 1 & 0 & 0 & 0 & 0 & 0 & 0 \\
 0 & 0 & 0 & 1 & 0 & 0 & 0 & 0 & 0 \\
 0 & 0 & 0 & 0 & \frac{3}{5} & -\frac{1}{5} \left(2 \sqrt{\frac{2}{3}}\right) & -2 \sqrt{\frac{2}{105}} & -\sqrt{\frac{2}{35}} & \sqrt{\frac{2}{5}} \\
 0 & 0 & 0 & 0 & -\frac{1}{5} \left(2 \sqrt{\frac{2}{3}}\right) & \frac{11}{15} & -\frac{4}{3 \sqrt{35}} & -\frac{2}{\sqrt{105}} & \frac{2}{\sqrt{15}} \\
 0 & 0 & 0 & 0 & -2 \sqrt{\frac{2}{105}} & -\frac{4}{3 \sqrt{35}} & \frac{17}{21} & -\frac{2}{7 \sqrt{3}} & \frac{2}{\sqrt{21}} \\
 0 & 0 & 0 & 0 & -\sqrt{\frac{2}{35}} & -\frac{2}{\sqrt{105}} & -\frac{2}{7 \sqrt{3}} & \frac{6}{7} & \frac{1}{\sqrt{7}} \\
 0 & 0 & 0 & 0 & -\sqrt{\frac{2}{5}} & -\frac{2}{\sqrt{15}} & -\frac{2}{\sqrt{21}} & -\frac{1}{\sqrt{7}} & 0 \\
\end{array}
\right).
\end{equation}

Let $x_i = \arctanh s_i = \beta(S_i, O, A_i)$ denote the hyperbolic distance of center of the model $A_1=(1,0,\dots, 0)$ to $S_i=(1,0,\dots,0,s_i)$ for $i\in\{0,5\}$, rotated in the case of $A_5$. 
If horoball $B_0$ is maximal $s_0=0$. If horoball $B_5$ is maximal then $s_5=\frac{3}{5}$.
These two maximal horoballs $B_0(\arctanh 0)$ and 
$B_5(\arctanh\frac{3}{5})$ are tangent to hyperfaces $[\Bu_0]$ and 
$[\Bu_5]$ respectively, and to each other at $H_5$. By two applications of Lemma \ref{lem:loc}, and Lemma \ref{lem:glob} the optimal backing packing density is 
$ \delta_{opt}(\Gamma)  = \frac{225}{8 \pi^4}$.
\end{proof}

\begin{corollary}
The optimal congruent ball packing density in $\bbH^8$ up to horoballs of the same type is bounded by
$ \tfrac{225}{8\pi^4} \leq \delta_{opt}(\overline{\mathbb{H}}^8) \leq 0.31114\dots$
\end{corollary}

\begin{landscape}

\begin{table}[h!]
\resizebox{\columnwidth}{!}{%
	\begin{tabular}{l|l|l|l|l}
		 \hline
		 \multicolumn{5}{c}{{\bf Coxeter Simplex Tilings} }\\
		\hline
		 Witt Symb. & $\overline{S}_8$ &  $\overline{Q}_8$ &  $\overline{T}_8$ &  $\overline{P}_8$ \\
		 \hline
		 \multicolumn{5}{c}{{\bf Vertices of Simplex} }\\
		 \hline
		 $A_0$ & $(1, 0, 0, 0, 0, 0, 0, 0, 1)$ & $(1, 0, 0, 0, 0, 0, 0, 0, 1)$ & $(1, 0, 0, 0, 0, 0, 0, 0, 1)$ & $(1, 0, 0, 0, 0, 0, 0, 0, 1)$  \\
		 $A_1$ & $(1, 0, 0, 0, 0, 0, 0, 0, 0)$ & $(1, 0, 0, 0, 0, 0, 0, 0, 0)$ & $(1, 0, 0, 0, 0, 0, 0, 0, 0)$ & $(1, 0, 0, 0, 0, 0, 0, 0, 0)$  \\
		 $A_2$ & $(1,0,0,0,0,0,0,\frac{1}{2},0)$ & $(1,0,0,0,0,0,0,\frac{1}{2},0)$ & $(1,0,0,0,0,0,0,\frac{1}{2},0)$ & $(1,0,0,0,0,0,0,\frac{\sqrt{7}}{4},0)$  \\
		 $A_3$ & $(1,0,0,0,0,0,\frac{\sqrt{2}}{4},\frac{1}{2},0)$ & $(1,0,0,0,0,0,\frac{\sqrt{2}}{4},\frac{1}{2},0)$ & $(1,0,0,0,0,0,\frac{\sqrt{3}}{6},\frac{1}{2},0)$ & $(1,0,0,0,0,0,\sqrt{\frac{3}{7}},\frac{3}{2 \sqrt{7}},0)$  \\
		 $A_4$ & $(1,0,0,0,0,\frac{\sqrt{2}}{4},\frac{\sqrt{2}}{4},\frac{1}{2},0)$ & $(1,0,0,0,0,\frac{\sqrt{2}}{4},\frac{\sqrt{2}}{4},\frac{1}{2},0)$ & $(1,0,0,0,0,\frac{\sqrt{6}}{12},\frac{\sqrt{3}}{6},\frac{1}{2},0)$ & $(1,0,0,0,0,\frac{\sqrt{\frac{5}{3}}}{2},\frac{5}{2 \sqrt{21}},\frac{5}{4 \sqrt{7}},0)$  \\
		 $A_5$ & $(1,0,0,0,\frac{\sqrt{2}}{4},\frac{\sqrt{2}}{4},\frac{\sqrt{2}}{4},\frac{1}{2},0)$ & $(1,0,0,0,\frac{1}{2},\frac{\sqrt{2}}{4},\frac{\sqrt{2}}{4},\frac{1}{2},0)$ & $(1,0,0,0,\frac{\sqrt{6}}{12},\frac{\sqrt{6}}{12},\frac{\sqrt{3}}{6},\frac{1}{2},0)$ & $(1,0,0,0,\sqrt{\frac{2}{5}},\frac{2}{\sqrt{15}},\frac{2}{\sqrt{21}},\frac{1}{\sqrt{7}},0)$  \\
		 $A_6$ & $(1,0,0,\frac{\sqrt{2}}{4},\frac{\sqrt{2}}{4},\frac{\sqrt{2}}{4},\frac{\sqrt{2}}{4},\frac{1}{2},0)$ & $(1,0,0,\frac{1}{2},\frac{\sqrt{2}}{4},\frac{\sqrt{2}}{4},\frac{\sqrt{2}}{4},\frac{1}{2},0)$ & $(1,0,0,\frac{\sqrt{3}}{6},\frac{\sqrt{6}}{12},\frac{\sqrt{6}}{12},\frac{\sqrt{3}}{6},\frac{1}{2},0)$ & $(1,0,0,\frac{\sqrt{\frac{3}{2}}}{2},\frac{3}{2 \sqrt{10}},\frac{\sqrt{\frac{3}{5}}}{2},\frac{\sqrt{\frac{3}{7}}}{2},\frac{3}{4 \sqrt{7}},0)$  \\
		 $A_7$ & $(1,0,\frac{\sqrt{2}}{4},\frac{\sqrt{2}}{4},\frac{\sqrt{2}}{4},\frac{\sqrt{2}}{4},\frac{\sqrt{2}}{4},\frac{1}{2},0)$ & $(1,0,\frac{1}{2},0,\frac{\sqrt{2}}{4},\frac{\sqrt{2}}{4},\frac{\sqrt{2}}{4},\frac{1}{2},0)$ & $(1,0,\frac{1}{2},\frac{\sqrt{3}}{6},\frac{\sqrt{6}}{12},\frac{\sqrt{6}}{12},\frac{\sqrt{3}}{6},\frac{1}{2},0)$ & $(1,0,\frac{1}{\sqrt{3}},\frac{1}{\sqrt{6}},\frac{1}{\sqrt{10}},\frac{1}{\sqrt{15}},\frac{1}{\sqrt{21}},\frac{1}{2 \sqrt{7}},0)$  \\
		 $A_8$ & $(1,\frac{1}{2},0,0,0,0,0,\frac{1}{2},0)$ & $(1,\frac{1}{2},0,0,0,0,0,\frac{1}{2},0)$ & $(1,\frac{1}{4},0,0,0,\frac{\sqrt{6}}{12},\frac{\sqrt{3}}{6},\frac{1}{2},0)$ & $(1,\frac{1}{2},\frac{1}{2 \sqrt{3}},\frac{1}{2 \sqrt{6}},\frac{1}{2 \sqrt{10}},\frac{1}{2 \sqrt{15}},\frac{1}{2 \sqrt{21}},\frac{1}{4 \sqrt{7}},0)$  \\
		 \hline
		 \multicolumn{5}{c}{{\bf The form $\mbox{\boldmath$u$}_i$ of sides opposite $A_i$ }}\\
		\hline
		 $\mbox{\boldmath$u$}_0$ & $(0, 0, 0, 0, 0, 0, 0, 0, 1)$ & $(0, 0, 0, 0, 0, 0, 0, 0, 1)$ & $(0, 0, 0, 0, 0, 0, 0, 0, 1)$ & $(0, 0, 0, 0, 0, 0, 0, 0, 1)$  \\
		 $\mbox{\boldmath$u$}_1$ & $(1, 0, 0, 0, 0, 0, 0, -2, -1)$ & $(1, 0, 0, 0, 0, 0, 0, -2, -1)$ & $(1, 0, 0, 0, 0, 0, 0, -2, -1)$ & $(1,-1,-\frac{1}{\sqrt{3}},-\frac{1}{\sqrt{6}},-\frac{1}{\sqrt{10}},-\frac{1}{\sqrt{15}},-\frac{1}{\sqrt{21}},-\frac{4}{\sqrt{7}},-1)$  \\
		 $\mbox{\boldmath$u$}_2$ & $(0,-1,0,0,0,0,-\sqrt{2},1,0)$ & $(0,-1,0,0,0,0,-\sqrt{2},1,0)$ & $(0,0,0,0,0,0,-\sqrt{3},1,0)$ & $(0,0,0,0,0,0,-\frac{\sqrt{3}}{2},1,0)$  \\
		 $\mbox{\boldmath$u$}_3$ & $(0, 0, 0, 0, 0, -1, 1, 0, 0)$ & $(0, 0, 0, 0, 0, -1, 1, 0, 0)$ & $(0,0,0,0,0,-\sqrt{2},1,0,0)$ & $(0,0,0,0,0,-\sqrt{\frac{5}{7}},1,0,0)$  \\
		 $\mbox{\boldmath$u$}_4$ & $(0, 0, 0, 0, -1, 1, 0, 0, 0)$ & $(0, 0, 0, 0, -1, 1, 0, 0, 0)$ & $(0,-\sqrt{\frac{2}{3}},0,0,-1,1,0,0,0)$ & $(0,0,0,0,-\sqrt{\frac{2}{3}},1,0,0,0)$  \\
		 $\mbox{\boldmath$u$}_5$ & $(0, 0, 0, -1, 1, 0, 0, 0, 0)$ & $(0,0,-\frac{1}{\sqrt{2}},-\frac{1}{\sqrt{2}},1,0,0,0,0)$ & $(0,0,0,-\frac{1}{\sqrt{2}},1,0,0,0,0)$ & $(0,0,0,-\sqrt{\frac{3}{5}},1,0,0,0,0)$  \\
		 $\mbox{\boldmath$u$}_6$ & $(0, 0, -1, 1, 0, 0, 0, 0, 0)$ & $(0, 0, 0, 1, 0, 0, 0, 0, 0)$ & $(0,0,-\frac{1}{\sqrt{3}},1,0,0,0,0,0)$ & $(0,0,-\frac{1}{\sqrt{2}},1,0,0,0,0,0)$  \\
		 $\mbox{\boldmath$u$}_7$ & $(0, 0, 1, 0, 0, 0, 0, 0, 0)$ & $(0, 0, 1, 0, 0, 0, 0, 0, 0)$ & $(0, 0, 1, 0, 0, 0, 0, 0, 0)$ & $(0,-\frac{1}{\sqrt{3}},1,0,0,0,0,0,0)$  \\
		 $\mbox{\boldmath$u$}_8$ & $(0, 1, 0, 0, 0, 0, 0, 0, 0)$ & $(0, 1, 0, 0, 0, 0, 0, 0, 0)$ & $(0, 1, 0, 0, 0, 0, 0, 0, 0)$ & $(0, 1, 0, 0, 0, 0, 0, 0, 0)$  \\
		 \hline
		 \multicolumn{5}{c}{{\bf Maximal horoball parameter $s_0$ }}\\
		\hline
		 $s_0$ & $0$ & $0$ & $0$ & $s_0=0, s_5=\frac{3}{5}$ \\
		\hline
		 \multicolumn{5}{c}{ {\bf Intersections $H_i = \mathcal{B}(A_0,s_0)\cap A_0A_i$ of horoballs with simplex edges}}\\
		\hline
		 $H_1$ & $(1, 0, 0, 0, 0, 0, 0, 0, 0)$ & $(1, 0, 0, 0, 0, 0, 0, 0, 0)$ & $(1, 0, 0, 0, 0, 0, 0, 0, 0)$ & $(1, 0, 0, 0, 0, 0, 0, 0, 0)$  \\
		 $H_2$ & $(1,0,0,0,0,0,0,\frac{4}{9},\frac{1}{9})$ & $(1,0,0,0,0,0,0,\frac{4}{9},\frac{1}{9})$ & $(1,0,0,0,0,0,0,\frac{4}{9},\frac{1}{9})$ & $(1,0,0,0,0,0,0,\frac{8 \sqrt{7}}{39},\frac{7}{39})$  \\
		 $H_3$ & $(1,0,0,0,0,0,\frac{4 \sqrt{2}}{19},\frac{8}{19},\frac{3}{19})$ & $(1,0,0,0,0,0,\frac{4 \sqrt{2}}{19},\frac{8}{19},\frac{3}{19})$ & $(1,0,0,0,0,0,\frac{\sqrt{3}}{7},\frac{3}{7},\frac{1}{7})$ & $(1,0,0,0,0,0,\frac{8 \sqrt{\frac{3}{7}}}{11},\frac{12}{11 \sqrt{7}},\frac{3}{11})$  \\
		 $H_4$ & $(1,0,0,0,0,\frac{\sqrt{2}}{5},\frac{\sqrt{2}}{5},\frac{2}{5},\frac{1}{5})$ & $(1,0,0,0,0,\frac{\sqrt{2}}{5},\frac{\sqrt{2}}{5},\frac{2}{5},\frac{1}{5})$ & $(1,0,0,0,0,\frac{4 \sqrt{\frac{2}{3}}}{19},\frac{8}{19 \sqrt{3}},\frac{8}{19},\frac{3}{19})$ & $(1,0,0,0,0,\frac{16 \sqrt{\frac{5}{3}}}{47},\frac{80}{47 \sqrt{21}},\frac{40}{47 \sqrt{7}},\frac{15}{47})$  \\
		 $H_5$ & $(1,0,0,0,\frac{4 \sqrt{2}}{21},\frac{4 \sqrt{2}}{21},\frac{4 \sqrt{2}}{21},\frac{8}{21},\frac{5}{21})$ & $(1,0,0,0,\frac{4 \sqrt{2}}{21},\frac{4 \sqrt{2}}{21},\frac{4 \sqrt{2}}{21},\frac{8}{21},\frac{5}{21})$ & $(1,0,0,0,\frac{2 \sqrt{6}}{29},\frac{2 \sqrt{6}}{29},\frac{4 \sqrt{3}}{29},\frac{12}{29},\frac{5}{29})$ & $(1,0,0,0,\frac{2 \sqrt{\frac{2}{5}}}{3},\frac{4}{3 \sqrt{15}},\frac{4}{3 \sqrt{21}},\frac{2}{3 \sqrt{7}},\frac{1}{3})$  \\
		 $H_6$ & $(1,0,0,\frac{2 \sqrt{2}}{11},\frac{2 \sqrt{2}}{11},\frac{2 \sqrt{2}}{11},\frac{2 \sqrt{2}}{11},\frac{4}{11},\frac{3}{11})$ & $(1,0,0,\frac{8}{23},\frac{4 \sqrt{2}}{23},\frac{4 \sqrt{2}}{23},\frac{4 \sqrt{2}}{23},\frac{8}{23},\frac{7}{23})$ & $(1,0,0,\frac{2}{5 \sqrt{3}},\frac{\sqrt{\frac{2}{3}}}{5},\frac{\sqrt{\frac{2}{3}}}{5},\frac{2}{5 \sqrt{3}},\frac{2}{5},\frac{1}{5})$ & $(1,0,0,\frac{8 \sqrt{6}}{47},\frac{24 \sqrt{\frac{2}{5}}}{47},\frac{16 \sqrt{\frac{3}{5}}}{47},\frac{16 \sqrt{\frac{3}{7}}}{47},\frac{24}{47 \sqrt{7}},\frac{15}{47})$  \\
		 $H_7$ & $(1,0,\frac{4 \sqrt{2}}{23},\frac{4 \sqrt{2}}{23},\frac{4 \sqrt{2}}{23},\frac{4 \sqrt{2}}{23},\frac{4 \sqrt{2}}{23},\frac{8}{23},\frac{7}{23})$ & $(1,0,\frac{8}{23},0,\frac{4 \sqrt{2}}{23},\frac{4 \sqrt{2}}{23},\frac{4 \sqrt{2}}{23},\frac{8}{23},\frac{7}{23})$ & $(1,0,\frac{4}{11},\frac{4}{11 \sqrt{3}},\frac{2 \sqrt{\frac{2}{3}}}{11},\frac{2 \sqrt{\frac{2}{3}}}{11},\frac{4}{11 \sqrt{3}},\frac{4}{11},\frac{3}{11})$ & $(1,0,\frac{8}{11 \sqrt{3}},\frac{4 \sqrt{\frac{2}{3}}}{11},\frac{4 \sqrt{\frac{2}{5}}}{11},\frac{8}{11 \sqrt{15}},\frac{8}{11 \sqrt{21}},\frac{4}{11 \sqrt{7}},\frac{3}{11})$  \\
		 $H_8$ & $(1,\frac{2}{5},0,0,0,0,0,\frac{2}{5},\frac{1}{5})$ & $(1,\frac{2}{5},0,0,0,0,0,\frac{2}{5},\frac{1}{5})$ & $(1,\frac{8}{39},0,0,0,\frac{8 \sqrt{\frac{2}{3}}}{39},\frac{16}{39 \sqrt{3}},\frac{16}{39},\frac{7}{39})$ & $(1,\frac{16}{39},\frac{16}{39 \sqrt{3}},\frac{8 \sqrt{\frac{2}{3}}}{39},\frac{8 \sqrt{\frac{2}{5}}}{39},\frac{16}{39 \sqrt{15}},\frac{16}{39 \sqrt{21}},\frac{8}{39 \sqrt{7}},\frac{7}{39})$  \\
		 \hline
		 \multicolumn{5}{c}{ {\bf Volume of maximal horoball piece }}\\
		\hline
		 $vol(\mathcal{B}_0 \cap \mathcal{F})$ & $(18063360 \sqrt{2})^{-1}$ & $(9031680 \sqrt{2})^{-1}$ & $162570240^{-1}$ &  $1128960^{-1}$  \\	
		 	\hline
		\multicolumn{5}{c}{ {\bf Optimal Packing Density}}\\
		\hline
		 $\delta_{opt}$ & $\frac{2025}{68 \sqrt{2} \pi ^4} \approx 0.21617\dots$ & $\frac{2025}{68 \sqrt{2} \pi ^4}$ & $\frac{225}{8 \pi ^4} \approx 0.28873\dots$ & $(\frac{9}{17}+\frac{8}{17})\frac{225}{8 \pi ^4}$  \\
		\hline
	\end{tabular}%
}
	\caption{Data for asymptotic Coxeter tilings of $\mathbb{H}^8$ in the Cayley-Klein ball model centered at $O=(1,0,0,0,0,0,0,0,0)$}
	\label{table:data_8dim}
\end{table}

\end{landscape}

\subsection{Case $n=9$ Dimensions}

\begin{theorem}
The optimal horoball packing density of Coxeter simplex tilings $\cT_\Gamma$, 
$\Gamma \in \Big\{ \overline{T}_9, \overline{Q}_9 \Big\}$
 is $\delta_{opt}(\Gamma) = \frac{1}{4 \zeta(5)}$, and for $\cT_{\overline{S}_9}$ is $\delta_{opt}(\overline{S}_9) = \frac{151}{1054 \zeta(5)}$.
\label{thm:9}
\end{theorem}

\begin{proof}
There are three cases for when $\cF_{\Gamma}$ has one, two, or three ideal vertices. 

Case 1: 
Coxeter simplex $\cF_{\overline{T}_9}$ in $\overline{\mathbb{H}}^9$ has one ideal vertex, the local optimal packing density follows from Lemma \ref{lem:loc}, and extends to the entire space by Lemma \ref{lem:glob}.
Our choice of vertices $A_i$, hyperplanes $\Bu_i$ opposite to $A_i$, optimal the horoball parameter $s$, horoball intersection points, and horoball piece volumes are given in Table \ref{table:data_9dim}.

Case 2: 
$\cF_{\overline{S}_9}$ has two ideal vertices,   
Table \ref{table:data_9dim} assignes coordinates, with ideal vertices 
at $A_0$ and $A_8$. 
We use two horoballs $B_0\left( \arctanh s_0 \right)$ and $B_8(\arctanh s_8)$ 
with parameters $s_0$ and $s_8$ at centered at $A_0$ and $A_8$ respectively.
Let $x_i = \arctanh s_i = \beta(S_i, O, A_i)$ denote the hyperbolic distance from the center of the model $A_1$ to $S_i=(1,0,\dots,0,s_i)$ for $i\in\{0,8\}$ (after rotation of $B_8$ as in Theorem 4).
If horoball $B_0$ is maximal then $s_0=0$. If horoball $B_8$ is maximal the $s_8=\frac{7}{9}$.
One can check that the two maximal type horoballs do not intersect, so
with two applications of Lemma \ref{lem:loc}, and then Lemma \ref{lem:glob} yields the optimal packing density 
$ \delta_{opt}(\overline{S}_9)  = \frac{151}{1054 \zeta (5)}$.

Case 3: 
Assign coordinates to the fundamental domain $\cF_{\overline{Q}_9}$ as in 
Table \ref{table:data_9dim}. The ideal vertices are 
$A_0$, $A_7$, and $A_8$.
Place horoballs $B_i(\arctanh s_i)$
with parameters
$s_i$ at $A_i$ for $i \in \{0,7,8\}$.
Let $x_i = \arctanh s_i = \beta(S_i, O, A_i)$ denote the hyperbolic distance from the center of the model 
$A_1$ to point $S_i=(1,0,\dots,0,s_i)$. $S_i \in B_i$ after the rotation of $A_i$ to $A_0$. 

If horoball $B_0$ is maximal then $s_0 = 0$, and the maximal tangent horoballs $B_7$ and $B_8$ have $s_7 = \frac{3}{5}$ and $s_8 = \frac{3}{5}$. 
If horoball $B_8$ is maximal type it is the same case up to symmetry, so it suffices to find the densities up to the midpoint of the allowed $s_i$ parameter range. 
If horoball $B_7$ is maximal its parameter is $s_7=\frac{3}{5}$ and the tangent 
maximal horoballs at $B_0$ and $B_8$ are respectively $s_0 = 0$ and $s_8=0$. 
Horoballs $B_0(\arctanh0)$ and 
$B_8(\arctanh \frac{3}{5})$ are tangent to hyperfaces $\Bu_0$ and 
$\Bu_8$ respectively. The densities of the extremal horoball arrangements 
are $\Theta = \frac{1}{4 \zeta (5)}$, in particular
\begin{equation}
\begin{split}
\Theta &= \delta_{s_0=0,s_7=\frac{3}{5},s_8=\frac{3}{5}}(\overline{Q}_9) \\
&= \frac{vol(\mathcal{B}_0(\arctanh0) \cap \cF_{\overline{Q}_9}) 
+ \sum_{i \in \{7,8\}} vol(\mathcal{B}_i(\arctanh\frac{3}{5})) \cap \cF_{\overline{Q}_9})
}{vol(\cF_{\overline{Q}_9})},\\
\Theta & = \delta_{s_0=\frac{3}{5},s_7=\frac{3}{5},s_8=0}(\overline{Q}_9) \\
& = \frac{ 
vol(\mathcal{B}_8(\arctanh 0) \cap \cF_{\overline{Q}_9})
+ \sum_{i \in \{0,7\}}
vol(\mathcal{B}_i(\arctanh\frac{3}{5}) \cap \cF_{\overline{Q}_9}) 
}{vol(\cF_{\overline{Q}_9})}.
\end{split}
\end{equation}
Next consider the horoball arrangements that continuously transition between the two extremal cases. Begin with the 
horoball arrangement with parameters $s_0=0$ and $s_8=\frac{3}{5}$, the horoballs 
$B_i(\arctanh s_i )$ where $i\in\{0,8\}$ are tangent. Define volumes 
$V_i(x) = vol(B_i(\arctanh s_i - x) \cap \cF_{\overline{Q}_9})$ for $i \in \{0,8\}$
  with $x \in [0,\arctanh\frac{3}{5}]$ where
$\arctanh \frac{3}{5}$ is the hyperbolic distance of 
$A_1$ and $S_i = (1,0,\dots,0,\frac{3}{5})$.
By formulas (2), (5), (6), and (7),  
$V_0(\arctanh 0 ) = \frac{1}{348364800}$, $V_7(\arctanh\frac{3}{5}) = \frac{1}{330301440}$  and $V_8(\arctanh\frac{3}{5}) = \frac{1}{89181388800}$. 
 By a weighted modification of Lemma \ref{lemma:szirmai}, 
\begin{equation}
\begin{split}
V(x) & = V_0(0) e^{-8x} + V_2\left(\arctanh\tfrac{3}{5}\right) + V_8\left(\arctanh\tfrac{3}{5}\right) e^{8x} \\
& = \frac{256  e^{-8x} + 270 +  e^{8x}}{89181388800}. 
\end{split}
\end{equation}
The densities of the intermediate cases between of the two 
extremal arrangements are given by
\begin{equation}
\begin{split}
\delta_{x}(\overline{Q}_9) &=
 \frac{vol(B_0(x) \cap \cF_{\overline{Q}_9}) 
+ vol(B_7(\arctanh \frac{3}{5} ) \cap \cF_{\overline{Q}_9}) + 
vol(B_8(\arctanh \frac{3}{5}-x ) \cap \cF_{\overline{Q}_9}) 
}{vol(\cF_{\overline{Q}_9})}\\
&= \left( \tfrac{256}{527} e^{-8x}+\tfrac{270}{527} +\tfrac{1}{527} e^{8x} \right) \Theta.
\end{split}
\end{equation}
where $x\in [0,\arctanh\frac{3}{5}]$.
Analysis of $\delta_{x}(\overline{Q}_9)$ shows that 
 its maxima are attained at the endpoints of the interval 
$[0,\arctanh \frac{3}{5}]$. In particular
\begin{equation}
\begin{split}
\delta_{x=\arctanh\frac{3}{5}}(\overline{Q}_9) 	&= \left( \tfrac{256}{527} e^{-8 \arctanh\tfrac{3}{5}} + \tfrac{270}{527} + \tfrac{1}{527} e^{8 \arctanh\frac{3}{5}} \right) 
\Theta
\\
& = \left( \frac{256}{527}\left(\frac{1-\tfrac{3}{5}}{1+\tfrac{3}{5}}\right)^4 + \frac{270}{527} +\frac{1}{527} \left(\frac{1+\tfrac{3}{5}}{1-\tfrac{3}{5}}\right)^4  \right) \Theta
\\
& = \left( \left(\tfrac{1}{4}\right)^4\tfrac{256}{527}+\tfrac{270}{527}+4^4\tfrac{1}{527} \right)\Theta
\\
& = \left( \tfrac{1}{527}+\tfrac{270}{527}+\tfrac{256}{527} \right)\Theta = \Theta.\\
\end{split}
\end{equation}

The numeric data of the optimal horoball packings are summarized in Table \ref{table:data_9dim}. 
The symmetry group $\Gamma_{\overline{Q}_9}$ extends the density from $\cF_{\overline{Q}_9}$ to the entire tiling.

\end{proof}

\begin{corollary}
The optimal congruent ball packing density in $\bbH^9$ up to horoballs of the same type is bounded by
$ \tfrac{1}{4 \zeta(5)} \leq \delta_{opt}(\overline{\mathbb{H}}^9) \leq 0.24285\dots.$
\end{corollary}

\begin{table}[h!]
\resizebox{\columnwidth}{!}{%
	\begin{tabular}{l|l|l|l}
		 \hline
		 \multicolumn{4}{c}{{\bf Coxeter Simplex Tilings} }\\
		\hline
		 Witt Symb. & $\overline{T}_9$ &  $\overline{S}_9$ &  $\overline{Q}_9$  \\
		 \hline
		 \multicolumn{4}{c}{{\bf Vertices of Simplex} }\\
		 \hline
		 $A_0$ & $(1, 0, 0, 0, 0, 0, 0, 0, 0, 1)$ & $(1, 0, 0, 0, 0, 0, 0, 0, 0, 1)$ & $(1, 0, 0, 0, 0, 0, 0, 0, 0, 1)$  \\
		 $A_1$ & $(1, 0, 0, 0, 0, 0, 0, 0, 0, 0)$ & $(1, 0, 0, 0, 0, 0, 0, 0, 0, 0)$ & $(1, 0, 0, 0, 0, 0, 0, 0, 0, 0)$  \\
		 $A_2$ & $(1,0,0,0,0,0,0,0,\frac{1}{2},0)$ & $(1,0,0,0,0,0,0,0,\frac{\sqrt{2}}{2},0)$ & $(1,0,0,0,0,0,0,0,\frac{\sqrt{3}}{3},0)$  \\
		 $A_3$ & $(1,0,0,0,0,0,0,\frac{\sqrt{3}}{6},\frac{1}{2},0)$ & $(1,0,0,0,0,0,0,\frac{\sqrt{6}}{6},\frac{\sqrt{2}}{2},0)$ & $(1,0,0,0,0,0,0,\frac{\sqrt{6}}{6},\frac{\sqrt{3}}{3},0)$  \\
		 $A_4$ & $(1,0,0,0,0,0,\frac{\sqrt{6}}{12},\frac{\sqrt{3}}{6},\frac{1}{2},0)$ & $(1,0,0,0,0,0,\frac{\sqrt{3}}{6},\frac{\sqrt{6}}{6},\frac{\sqrt{2}}{2},0)$ & $(1,0,0,0,0,0,\frac{\sqrt{10}}{10},\frac{\sqrt{6}}{6},\frac{\sqrt{3}}{3},0)$  \\
		 $A_5$ & $(1,0,0,0,0,\frac{\sqrt{10}}{20},\frac{\sqrt{6}}{12},\frac{\sqrt{3}}{6},\frac{1}{2},0)$ & $(1,0,0,0,0,\frac{\sqrt{5}}{10},\frac{\sqrt{3}}{6},\frac{\sqrt{6}}{6},\frac{\sqrt{2}}{2},0)$ & $(1,0,0,0,0,\frac{\sqrt{15}}{15},\frac{\sqrt{10}}{10},\frac{\sqrt{6}}{6},\frac{\sqrt{3}}{3},0)$  \\
		 $A_6$ & $(1,0,0,0,\frac{\sqrt{15}}{30},\frac{\sqrt{10}}{20},\frac{\sqrt{6}}{12},\frac{\sqrt{3}}{6},\frac{1}{2},0)$ & $(1,0,0,0,\frac{\sqrt{30}}{30},\frac{\sqrt{5}}{10},\frac{\sqrt{3}}{6},\frac{\sqrt{6}}{6},\frac{\sqrt{2}}{2},0)$ & $(1,0,0,0,\frac{\sqrt{3}}{6},\frac{\sqrt{15}}{15},\frac{\sqrt{10}}{10},\frac{\sqrt{6}}{6},\frac{\sqrt{3}}{3},0)$  \\
		 $A_7$ & $(1,0,0,\frac{\sqrt{3}}{12},\frac{\sqrt{15}}{30},\frac{\sqrt{10}}{20},\frac{\sqrt{6}}{12},\frac{\sqrt{3}}{6},\frac{1}{2},0)$ & $(1,0,0,\frac{\sqrt{6}}{12},\frac{\sqrt{30}}{30},\frac{\sqrt{5}}{10},\frac{\sqrt{3}}{6},\frac{\sqrt{6}}{6},\frac{\sqrt{2}}{2},0)$ & $(1,0,0,\frac{1}{2},\frac{\sqrt{3}}{6},\frac{\sqrt{15}}{15},\frac{\sqrt{10}}{10},\frac{\sqrt{6}}{6},\frac{\sqrt{3}}{3},0)$  \\
		 $A_8$ & $(1,0,\frac{1}{4},\frac{\sqrt{3}}{12},\frac{\sqrt{15}}{30},\frac{\sqrt{10}}{20},\frac{\sqrt{6}}{12},\frac{\sqrt{3}}{6},\frac{1}{2},0)$ & $(1,0,\frac{\sqrt{2}}{4},\frac{\sqrt{6}}{12},\frac{\sqrt{30}}{30},\frac{\sqrt{5}}{10},\frac{\sqrt{3}}{6},\frac{\sqrt{6}}{6},\frac{\sqrt{2}}{2},0)$ & $(1, 0, 1, 0, 0, 0, 0, 0, 0, 0)$  \\
		 $A_9$ & $(1,\frac{1}{6},0,0,\frac{\sqrt{15}}{30},\frac{\sqrt{10}}{20},\frac{\sqrt{6}}{12},\frac{\sqrt{3}}{6},\frac{1}{2},0)$ & $(1,\frac{\sqrt{2}}{6},0,0,\frac{\sqrt{30}}{30},\frac{\sqrt{5}}{10},\frac{\sqrt{3}}{6},\frac{\sqrt{6}}{6},\frac{\sqrt{2}}{2},0)$ & $(1,\frac{1}{3},0,0,0,\frac{1}{\sqrt{15}},\frac{1}{\sqrt{10}},\frac{1}{\sqrt{6}},\frac{1}{\sqrt{3}},0)$  \\
		 \hline
		 \multicolumn{4}{c}{{\bf The form $\mbox{\boldmath$u$}_i$ of sides opposite $A_i$ }}\\
		\hline
		 $\mbox{\boldmath$u$}_0$ & $(0, 0, 0, 0, 0, 0, 0, 0, 0, 1)$ & $(0, 0, 0, 0, 0, 0, 0, 0, 0, 1)$ & $(0, 0, 0, 0, 0, 0, 0, 0, 0, 1)$  \\
		 $\mbox{\boldmath$u$}_1$ & $(1, 0, 0, 0, 0, 0, 0, 0, -2, -1)$ & $(1,0,0,0,0,0,0,0,-\sqrt{2},-1)$ & $(1, 0, -1, 0, 0, 0, 0, 0, -\sqrt{3}, -1)$  \\
		 $\mbox{\boldmath$u$}_2$ & $(0, 0, 0, 0, 0, 0, 0, -\sqrt{3}, 1, 0)$ & $(0,0,0,0,0,0,0,-\sqrt{3},1,0)$ & $(0,0,0,0,0,0,0,-\sqrt{2},1,0)$  \\
		 $\mbox{\boldmath$u$}_3$ & $(0,0,0,0,0,0,-\sqrt{2},1,0,0)$ & $(0,0,0,0,0,0,-\sqrt{2},1,0,0)$ & $(0,0,0,0,0,0,-\sqrt{\frac{5}{3}},1,0,0)$  \\
		 $\mbox{\boldmath$u$}_4$ & $(0,0,0,0,0,-\sqrt{\frac{5}{3}},1,0,0,0)$ & $(0,0,0,0,0,-\sqrt{\frac{5}{3}},1,0,0,0)$ & $(0,0,0,0,0,-\sqrt{\frac{3}{2}},1,0,0,0)$  \\
		 $\mbox{\boldmath$u$}_5$ & $(0,0,0,0,-\sqrt{\frac{3}{2}},1,0,0,0,0)$ & $(0,0,0,0,-\sqrt{\frac{3}{2}},1,0,0,0,0)$ & $(0,-\sqrt{\frac{3}{5}},0,0,-\frac{2}{\sqrt{5}},1,0,0,0,0)$  \\
		 $\mbox{\boldmath$u$}_6$ & $(\left.0,-\sqrt{\frac{3}{5}},0,-\frac{2}{\sqrt{5}},1,0,0,0,0,0\right\})$ & $(\left.0,-\sqrt{\frac{3}{5}},0,-\frac{2}{\sqrt{5}},1,0,0,0,0,0\right\})$ & $(0,0,0,-\frac{1}{\sqrt{3}},1,0,0,0,0,0)$  \\
		 $\mbox{\boldmath$u$}_7$ & $(0,0,-\frac{1}{\sqrt{3}},1,0,0,0,0,0,0)$ & $(0,0,-\frac{1}{\sqrt{3}},1,0,0,0,0,0,0)$ & $(0, 0, 0, 1, 0, 0, 0, 0, 0, 0)$  \\
		 $\mbox{\boldmath$u$}_8$ & $(0, 0, 1, 0, 0, 0, 0, 0, 0, 0)$ & $(0, 0, 1, 0, 0, 0, 0, 0, 0, 0)$ & $(0, 0, 1, 0, 0, 0, 0, 0, 0, 0)$  \\
		 $\mbox{\boldmath$u$}_9$ & $(0, 1, 0, 0, 0, 0, 0, 0, 0, 0)$ & $(0, 1, 0, 0, 0, 0, 0, 0, 0, 0)$ & $(0, 1, 0, 0, 0, 0, 0, 0, 0, 0)$  \\
		 \hline
		 \multicolumn{4}{c}{{\bf Maximal horoball parameters $s_i$ }}\\
		\hline
		 $s_i$ & $s_0=0$ & $s_0=0, s_8=7/9$ & $s_0=0, s_7=3/5, s_8=0$ \\
		\hline
		 \multicolumn{4}{c}{ {\bf Intersections $H_i = \mathcal{B}(A_0,s_0)\cap A_0A_i$ of horoballs with simplex edges}}\\
		\hline
		 $H_1$ & $(1,0,0,0,0,0,0,0,0,0)$ & $(1,0,0,0,0,0,0,0,0,0)$  & $(1,0,0,0,0,0,0,0,0,0)$    \\
		 $H_2$ & $(1,0,0,0,0,0,0,0,\frac{4}{9},\frac{1}{9})$ & $(1,0,0,0,0,0,0,0,\frac{2 \sqrt{2}}{5},\frac{1}{5})$ & $(1,0,0,0,0,0,0,0,\frac{2 \sqrt{3}}{7},\frac{1}{7})$    \\
		 $H_3$ & $(1,0,0,0,0,0,0,\frac{\sqrt{3}}{7},\frac{3}{7},\frac{1}{7})$ & $(1,0,0,0,0,0,0,\frac{\sqrt{\frac{3}{2}}}{4},\frac{3}{4 \sqrt{2}},\frac{1}{4})$ & $(1,0,0,0,0,0,0,\frac{2 \sqrt{\frac{2}{3}}}{5},\frac{4}{5 \sqrt{3}},\frac{1}{5})$   \\
		 $H_4$ & $(1,0,0,0,0,0,\frac{4 \sqrt{\frac{2}{3}}}{19},\frac{8}{19 \sqrt{3}},\frac{8}{19},\frac{3}{19})$ & $(1,0,0,0,0,0,\frac{4}{11 \sqrt{3}},\frac{4 \sqrt{\frac{2}{3}}}{11},\frac{4 \sqrt{2}}{11},\frac{3}{11})$ & $(1,0,0,0,0,0,\frac{\sqrt{10}}{13},\frac{5 \sqrt{\frac{2}{3}}}{13},\frac{10}{13 \sqrt{3}},\frac{3}{13})$   \\
		 $H_5$ & $(1,0,0,0,0,\frac{\sqrt{\frac{5}{2}}}{12},\frac{5}{12 \sqrt{6}},\frac{5}{12 \sqrt{3}},\frac{5}{12},\frac{1}{6})$ & $(1,0,0,0,0,\frac{\sqrt{5}}{14},\frac{5}{14 \sqrt{3}},\frac{5}{7 \sqrt{6}},\frac{5}{7 \sqrt{2}},\frac{2}{7})$ & $(1,0,0,0,0,\frac{\sqrt{\frac{3}{5}}}{4},\frac{3}{4 \sqrt{10}},\frac{\sqrt{\frac{3}{2}}}{4},\frac{\sqrt{3}}{4},\frac{1}{4})$   \\
		 $H_6$ & $(1,0,0,0,\frac{4 \sqrt{\frac{3}{5}}}{29},\frac{6 \sqrt{\frac{2}{5}}}{29},\frac{2 \sqrt{6}}{29},\frac{4 \sqrt{3}}{29},\frac{12}{29},\frac{5}{29})$ & $(1,0,0,0,\frac{2 \sqrt{\frac{6}{5}}}{17},\frac{6}{17 \sqrt{5}},\frac{2 \sqrt{3}}{17},\frac{2 \sqrt{6}}{17},\frac{6 \sqrt{2}}{17},\frac{5}{17})$ & $(1,0,0,0,\frac{4}{11 \sqrt{3}},\frac{8}{11 \sqrt{15}},\frac{4 \sqrt{\frac{2}{5}}}{11},\frac{4 \sqrt{\frac{2}{3}}}{11},\frac{8}{11 \sqrt{3}},\frac{3}{11})$   \\
		 $H_7$ & $(1,0,0,\frac{8}{39 \sqrt{3}},\frac{16}{39 \sqrt{15}},\frac{8 \sqrt{\frac{2}{5}}}{39},\frac{8 \sqrt{\frac{2}{3}}}{39},\frac{16}{39 \sqrt{3}},\frac{16}{39},\frac{7}{39})$ & $(1,0,0,\frac{4 \sqrt{\frac{2}{3}}}{23},\frac{8 \sqrt{\frac{2}{15}}}{23},\frac{8}{23 \sqrt{5}},\frac{8}{23 \sqrt{3}},\frac{8 \sqrt{\frac{2}{3}}}{23},\frac{8 \sqrt{2}}{23},\frac{7}{23})$ & $(1,0,0,\frac{1}{3},\frac{1}{3 \sqrt{3}},\frac{2}{3 \sqrt{15}},\frac{\sqrt{\frac{2}{5}}}{3},\frac{\sqrt{\frac{2}{3}}}{3},\frac{2}{3 \sqrt{3}},\frac{1}{3})$   \\
		 $H_8$ & $(1,0,\frac{1}{5},\frac{1}{5 \sqrt{3}},\frac{2}{5 \sqrt{15}},\frac{\sqrt{\frac{2}{5}}}{5},\frac{\sqrt{\frac{2}{3}}}{5},\frac{2}{5 \sqrt{3}},\frac{2}{5},\frac{1}{5})$ & $(1,0,\frac{1}{3 \sqrt{2}},\frac{1}{3 \sqrt{6}},\frac{\sqrt{\frac{2}{15}}}{3},\frac{1}{3 \sqrt{5}},\frac{1}{3 \sqrt{3}},\frac{\sqrt{\frac{2}{3}}}{3},\frac{\sqrt{2}}{3},\frac{1}{3})$ & $(1,0,\frac{2}{3},0,0,0,0,0,0,\frac{1}{3})$   \\
		 $H_9$ & $(1,\frac{3}{22},0,0,\frac{3 \sqrt{\frac{3}{5}}}{22},\frac{9}{22 \sqrt{10}},\frac{3 \sqrt{\frac{3}{2}}}{22},\frac{3 \sqrt{3}}{22},\frac{9}{22},\frac{2}{11})$ & $(1,\frac{3}{13 \sqrt{2}},0,0,\frac{3 \sqrt{\frac{3}{10}}}{13},\frac{9}{26 \sqrt{5}},\frac{3 \sqrt{3}}{26},\frac{3 \sqrt{\frac{3}{2}}}{13},\frac{9}{13 \sqrt{2}},\frac{4}{13})$ & $(1,\frac{6}{25},0,0,0,\frac{6 \sqrt{\frac{3}{5}}}{25},\frac{9 \sqrt{\frac{2}{5}}}{25},\frac{3 \sqrt{6}}{25},\frac{6 \sqrt{3}}{25},\frac{7}{25})$   \\
		 \hline
		 \multicolumn{4}{c}{ {\bf Volume of maximal horoball piece }}\\
		\hline
		 $vol(\mathcal{B}_0 \cap \mathcal{F})$ & $89181388800^{-1}$ & $5573836800^{-1}$  & $348364800^{-1}$  \\
		 		\hline
		\multicolumn{4}{c}{ {\bf Optimal Packing Density}}\\
		\hline
		 $\delta_{opt}$ & $\frac{1}{4 \zeta (5)} \approx 0.24109\dots$ & $(\frac{135}{151}+\frac{16}{151})\frac{151}{1054 \zeta (5)} \approx 0.138162\dots$ & $(\frac{256}{527}+\frac{270}{527}+\frac{1}{527})\frac{1}{4 \zeta (5)}=\frac{1}{4 \zeta (5)} \approx 0.24109\dots$  \\
		\hline
	\end{tabular}%
}
	\caption{Data for asymptotic Coxeter tilings of $\mathbb{H}^9$ in the Cayley-Klein ball model centered at $O=(1,0,0,0,0,0,0,0,0,0)$}
	\label{table:data_9dim}
\end{table}

\newpage


\end{document}